\documentclass[11pt, a4paper]{article}

\usepackage{amsmath}
\usepackage{amssymb}
\usepackage{amsthm}
\usepackage{authblk}
\usepackage{dsfont}
\usepackage[T1]{fontenc}
\usepackage{geometry}
\usepackage{mathtools}
\usepackage{natbib}
\usepackage{url}
\usepackage[dvipsnames]{xcolor}

\usepackage[colorlinks=true,allcolors=blue]{hyperref}
\usepackage{cleveref}
\usepackage{xr-hyper}

\numberwithin{equation}{section}

\theoremstyle{plain}
\newtheorem{theorem}{Theorem}[section]
\newtheorem{lemma}{Lemma}[section]
\newtheorem{proposition}{Proposition}[section]
\newtheorem{remark}{Remark}[section]

\DeclareMathOperator{\C}{\mathbb{C}}
\DeclareMathOperator{\N}{\mathbb{N}}
\DeclareMathOperator{\R}{\mathbb{R}}
\DeclareMathOperator{\Z}{\mathbb{Z}}

\DeclareMathOperator{\e}{\mathrm{e}}
\DeclareMathOperator{\imag}{\mathrm{i}}

\DeclareMathOperator{\Cov}{\mathrm Cov}
\DeclareMathOperator{\Exp}{\mathbb{E}}
\DeclareMathOperator{\Prob}{\mathbb{P}}

\begin{document}

\title{Self-normalization for Spectral Density Integrals}

\author[]{Holger Dette}
\author[]{Sebastian K\"uhnert}

\affil[]{Department of Mathematics, Ruhr University Bochum, 44780 Bochum, Germany, \textsuperscript{}\href{mailto:holger.dette@rub.de}{holger.dette@rub.de}, \href{mailto:sebastian.kuehnert@rub.de}{sebastian.kuehnert@rub.de}}

\date{\today}

\maketitle

\begin{abstract}
Integrals of spectral densities are frequently used to summarize spectral characteristics of linear processes. This work studies self-normalization for estimators of such integrals based on sequential periodograms and establishes weak convergence of the corresponding processes. For linear functionals of the spectral density, self-normalization yields pivotal limiting distributions that are free of unknown spectral quantities. For non-linear functionals, however, additional components with distinct covariance structures may arise in the limiting process. We demonstrate this phenomenon for the integrated squared spectral density.
\end{abstract} 



\section{Introduction}\label{Sec1}
In frequency-domain time series analysis, see, for instance, \cite{Bloomfield2000,Priestley1981} for comprehensive overviews, integrals of the form
\begin{align}\label{eq1}
    \Phi=\int_{-\pi}^{\pi} w(\lambda)\phi(f(\lambda))\,\mathrm d\lambda,
\end{align}
where $f$ denotes the spectral density and $w$ and $\phi$ are given functions, play a fundamental role, with applications across a wide range of fields, including physiology and neuroscience \citep{Akselrod1981,Buzsaki2006}, signal processing \citep{Hayes1996}, engineering \citep{Vanmarcke1972}, astrophysics \citep{VaughanEtal2003}, and economics \citep{LevyDezhbakhsh2003}. From a theoretical perspective, the estimation of such integrals has received considerable attention; see, for instance, \cite{GiraitisKoul2013,Ginovyan1988,HarizEtAl2026,Taniguchi1980}. Multivariate extensions have also been investigated in \cite{Akashietal2021,Dahlhaus1985,Dahlhaus1988b}, with the latter author providing particularly general results on multivariate spectral distribution functions under mild conditions and, in the more recent work, establishing process-level convergence indexed by classes of functions. The weak limits obtained in these works generally depend on unknown spectral quantities, such as integrals involving powers of the spectral density and the fourth-order cumulant spectral density \citep[cf.][]{Brillinger2001}. An exception is \cite{ShaoZhang2010}, who mention a sequential result for (vectors of) integrals of the form $\int_0^x f(\lambda)\,\mathrm d\lambda$ under implicit conditions.

This article investigates self-normalized sequential estimators of the spectral density integrals $\Phi$ in \eqref{eq1} for real-valued linear processes, with particular emphasis on whether self-normalization yields pivotal (parameter-free) limits, thereby avoiding the estimation of intricate spectral quantities. The underlying ideas of self-normalization can be found, for instance, in \cite{Lobato2001,Shao2015}, while more recent applications to a variety of time series inference problems include \cite{Bastian2026,ChengChan2024,DetteKuehnert2025,HongEtal2024,SunZhuLinton2025}. We establish results for both linear and quadratic $\phi$ under comparatively mild assumptions on the weight function $w$, primarily requiring H\"older continuity. In particular, we obtain pivotal weak limits in the linear setting, enabling inference without estimating unknown spectral quantities, while we observe that the situation is generally different in the non-linear case. While the form of these limits is arguably intuitive, their rigorous derivation is technically demanding. In particular, the sequential setting leads to substantially more intricate covariance structures and cumulant calculations, as it precludes several simplifying identities available in the non-sequential setting when summing over the entire grid of Fourier frequencies. The proofs are therefore technically demanding. 

The remainder of this paper is organized as follows. Our results are presented in Section~\ref{Sec2}, followed by a discussion in Section~\ref{Sec3}. The proofs are given in Section~\ref{Sec4}, with auxiliary results used therein collected in Appendix~\ref{Sec5}. 

\section{Results}\label{Sec2}

The focus is on sequential estimation of the spectral density integral $\Phi$ in \eqref{eq1} for certain symmetric functions $w\colon[-\pi,\pi]\to\R$, considering both $\phi(x)=x$ and $\phi(x)=x^2$. We assume that the observations arise from the linear process
\begin{align}\label{eq2}
    X_k=\sum^\infty_{j=-\infty}\psi_j\varepsilon_{k-j},
    \qquad
    0 < \sum^\infty_{j=-\infty}|j|^\alpha|\psi_j|<\infty,
    \quad \alpha \in (1/2,1],  
\end{align}
where $(\varepsilon_k)_{k\in\mathbb Z}$ is a Gaussian white noise with unit variance. As a consequence, $(X_k)$ is a stationary, centered Gaussian process and its spectral density has the form
\begin{align}
    f(\lambda)=\frac{1}{2\pi}\sum^\infty_{k=-\infty}\operatorname{Cov}(X_0,X_k)\mathrm e^{-\mathrm i\lambda k} \geq 0, \qquad \lambda\in[-\pi,\pi],
\end{align}
and is H\"older continuous of order $\alpha \in (1/2,1]$ by the summability condition in \eqref{eq2}. Our proposed estimators are based on the \emph{sequential periodogram}
\begin{align}\label{eq3}
    I_N^s(\lambda)\coloneqq\frac{1}{2\pi sN}\bigg|\sum_{k=1}^{\lfloor sN\rfloor}X_k \e^{-\mathrm{i}k\lambda}\bigg|^2, \qquad s \in \mathbb D \coloneqq [1/2,1] , ~ \lambda \in [-\pi, \pi],
\end{align}
where $X_1,\ldots,X_N$ are consecutive observations of the process in \eqref{eq2}. For $s=1,$ this refers to the common \emph{periodogram} $I_N(\lambda),$ which is an inconsistent but asymptotically unbiased estimator for $f(\lambda)$ \citep[see, e.g.,][]{BrockwellDavis1991}. 

\subsection{Linear case}\label{Sec21}
Let $\phi(x)=x$. Although $I_N(\lambda)$ is an inconsistent estimator for $f(\lambda)$, the \emph{weighted integrated periodogram} $\Phi_N=\frac{4\pi}{N}\sum_{j=1}^{\lfloor N/2\rfloor}w(\lambda_j)I_N(\lambda_j),$ where $\lambda_j=\frac{2\pi j}{N}$ are the \emph{Fourier frequencies}, consistently estimates $\Phi$. Heuristically, this follows from $\Exp(I_N(\lambda))\to f(\lambda)$ and Riemann-sum approximation. Unlike its integral counterpart $\Phi'_N = \int_{-\pi}^{\pi} w(\lambda)I_N(\lambda)\,\mathrm{d}\lambda$, which is asymptotically equivalent to $\Phi_N$ in the sense that $\Phi_N-\Phi'_N=o_{\Prob}(N^{-1/2})$ \citep[cf.][]{DeoChen2000}, $\Phi_N$ is directly implementable. Moreover, it naturally motivates the \emph{weighted integrated sequential periodogram}
\begin{align}\label{eq4}
    \Phi_N(s)\coloneqq\frac{4\pi}{N}\sum_{j=1}^{\lfloor N/2\rfloor}w(\lambda_j)I_N^s(\lambda_j),
    \qquad s\in\mathbb D.
\end{align}
Using self-normalization in the spirit of, e.g., \cite{Shao2015}, we obtain the following weak limit (weak convergence is denoted by $\stackrel{d}{\rightarrow}$).
\begin{theorem}\label{Thm1}
Let the function $w : [-\pi, \pi] \to \R$ be symmetric, H\"older continuous of order $\beta \in (1/2,1],$ and non-zero (Lebesgue-)almost everywhere. Then, as $N\to\infty,$
\[
    \frac{\Phi_N-\Phi}
    {\left(\displaystyle\int_{\mathbb D}
    \big(\Phi_N(s)-\Phi_N(1)\big)^2\,\mathrm{d}s\right)^{\!1/2}}
    ~\stackrel{d}{\longrightarrow}~
    \frac{\mathbb B(1)}
    {\left(\displaystyle\int_{\mathbb D}
    \frac{1}{s^2}\big(\mathbb B(s)-s\mathbb B(1)\big)^2
    \,\mathrm{d}s\right)^{\!1/2}}\,,
\]
where $\mathbb B = \{\mathbb B(s)\}_{s \in \mathbb D}$ is a Brownian motion on $\mathbb D,$ and the denominator on the right-hand side is almost surely non-zero \citep{QuanzDette2025}.
\end{theorem}
Importantly, the limiting distribution is pivotal, as it does not depend on any unknown spectral quantities. Its quantiles, while not available in closed form, can be simulated with high accuracy.

\begin{remark}
The assumptions on $w$ in Theorem \ref{Thm1} are relatively mild. Symmetry is imposed mainly to simplify the derivations, while H\"older continuity of order $\beta\in (1/2,1]$ ensures that the Riemann-sum approximation errors decay sufficiently fast. Under these conditions, we obtain the weak convergence
\begin{align}\label{eq5}
    \boldsymbol{\mathcal F}_N
    =
    \big\{\mathcal F_N(s) \big\}_{s\in\mathbb D}
    \coloneqq 
    \Big\{\sqrt{N}(\Phi_N(s)-\Phi)\Big\}_{s\in\mathbb D}
    \,\stackrel{d}{\longrightarrow}\,
    \boldsymbol{\mathcal F}
    = 
    \big\{\mathcal F(s) \big\}_{s\in\mathbb D},
\end{align}
where convergence takes place in the space $\ell^\infty(\mathbb D)$ of bounded real-valued functions on $\mathbb D$ \citep{VaartWellner2023}. Here, $\boldsymbol{\mathcal F}$ denotes the centered Gaussian process defined by
\begin{align}
    \mathcal F(s)
    \coloneqq
    \frac{\tau_{w_f}}{s}\,\mathbb B(s),
    \qquad s\in\mathbb D,
\end{align}
with
\begin{align}\label{eq6}
    \tau_{w_f}
    \coloneqq
    \bigg(
        4\pi\int_{-\pi}^{\pi}
        w^2(\lambda)f^2(\lambda)\,\mathrm{d}\lambda
    \bigg)^{\!1/2},
\end{align}
and $\mathbb B=\{\mathbb B(s)\}_{s\in\mathbb D}$ denoting a Brownian motion. Since $w$ is assumed to be non-zero (Lebesgue-)almost everywhere and $f\not\equiv 0$ under our assumptions, we have $\tau_{w_f}>0$. Consequently, the self-normalized limit is well-defined, with the nuisance parameter $\tau_{w_f}$ canceling out. We note that, in principle, it would suffice for $w$ to be non-zero almost-everywhere on a non-degenerate interval belonging to the support of $f.$ Since $f$ is continuous and $f\not\equiv 0$, such a non-degenerate interval always exists. Nevertheless, assuming that \(w\) is non-zero almost everywhere on \([-\pi,\pi]\) is natural from a statistical perspective, as the true spectral density \(f\), and hence the location of such an interval, is typically unknown.
\end{remark}

\subsection{Quadratic case}\label{Sec22}
This section considers the integral $\Phi$ in \eqref{eq1} for $\phi(x)=x^2$ and $w\equiv1,$ which we denote as $\Psi.$ We define its sequential estimator by
\begin{align}\label{eq7}
    \Psi_N(s)\coloneqq\frac{4\pi}{N}\sum_{j=1}^{\lfloor N/2\rfloor}I_N^s(\lambda_j)I_N^s(\lambda_{j-1}), \qquad s\in\mathbb D,
\end{align}
with the corresponding sequential population counterpart given by
\begin{align}\label{eq8}
    \Psi(s)\coloneqq\bigg(1+\frac{\sin^2(\pi s)}{\pi^2s^2}\bigg)\Psi, \qquad s\in\mathbb D.
\end{align}

\begin{proposition}\label{Prop1} As $N \to \infty,$ we have
\begin{align}
    \boldsymbol{\mathcal{G}}_N = \big\{
    \mathcal{G}_N(s)
    \big\}_{s\in\mathbb{D}} \coloneqq \Big\{\sqrt{N}\,\big(\Psi_N(s) - \Psi(s)\big)\Big\}_{s\in\mathbb D}
    \stackrel{d}{\longrightarrow}\, \boldsymbol{\mathcal{G}} = \big\{
    \mathcal{G}(s)
    \big\}_{s\in\mathbb{D}}\,,
\end{align}
in $\ell^\infty(\mathbb D),$ where $\boldsymbol{\mathcal G}$ is the centered Gaussian process defined by
\begin{align}\label{eq9}
    \mathcal{G}(s)
    =
    \tau_{f^2}\bigg(\frac{1}{s}\mathbb{B}_1(s)
    +
    \frac{1}{2s^2}
    \mathbb{B}_2(s^2k(s))\bigg), \qquad s\in\mathbb{D}\,.
\end{align}
Here, $\tau_{f^2}$ denotes the constant defined by
\begin{align}
    \tau_{f^2}
    \coloneqq
    \bigg(16\pi\int^\pi_{-\pi}f^4(\lambda)\,\mathrm{d}\lambda\bigg)^{\!1/2}\!,
\end{align}
and $\mathbb{B}_1$ and $\mathbb{B}_2$ are independent standard Brownian motions on $[0,1],$ and
\begin{align}\label{eq10}
    k(s)
    \coloneqq 
    \frac{2s^3 + (2s-1)^3}{3s^2}
    +
    \frac{
    20\pi s
    -4\pi(2s-1)\bigl(\cos(4\pi s)+4\cos(2\pi s)\bigr)
    +8\sin(2\pi s)
    +\sin(4\pi s)
    }{
    16\pi^3s^2
    }.
\end{align}
\end{proposition}

The limiting process $\boldsymbol{\mathcal G}$ in Proposition~\ref{Prop1}, upon setting $w=2f$, coincides with that in Theorem~\ref{Thm1}, except for the additional Brownian motion term $\frac{1}{2s^2}\mathbb{B}_2(s^2k(s))$. This term arises from additional contributions to the covariances of products of sequential periodograms. We note that $\boldsymbol{\mathcal G}$ is well-defined, since the map $s\mapsto s^2k(s)\in[0,1]$, appearing as the argument of $\mathbb B_2$, is strictly increasing. Similarly, the factor appearing in the definition of the population counterpart $\Psi(s)$ arises from the first-order structure specific to products of sequential periodograms. At $s=1$, however, $\Psi(1)=\Psi$, thus recovering the non-sequential setting \citep[cf., e.g.,][]{DetteKinsvaterVetter2011}.

Notably, since $\boldsymbol{\mathcal G}$ is the sum of two processes with distinct covariance structures, self-normalization as in Theorem~\ref{Thm1} leads to a non-pivotal limit. However, since, unlike in the linear case, $\Psi_N$ and its integral counterpart are not asymptotically equivalent \citep{DeoChen2000}, one might ask whether a pivotal weak limit for the integrated squared spectral density can instead be obtained by replacing $\Psi_N(s)$ in \eqref{eq7} with a sequential version of the corresponding integral statistic. To investigate this possibility, consider
\begin{align}
    \Psi'_N(s) \coloneqq \int_{-\pi}^{\pi}\big(I_N^s(\lambda)\big)^2\,\mathrm{d}\lambda, \qquad s\in\mathbb D.
\end{align}
The following result derives the limiting covariance structure of $\Psi'_N(s)$ in the particular case where the underlying process is a Gaussian white noise with unit variance. 

\begin{proposition}\label{Prop2} Let $(X_k)$ be a Gaussian white noise with unit variance. Then,
    \begin{align}
        N\Cov\big(\Psi'_N(s), \Psi'_N(t)\big) \stackrel{N \to \infty}{\longrightarrow}
    \frac{4}{\pi^2}
    \bigg(\frac{2}{\max(s,t)}
    + \frac{\min(s,t)}{3\max^2(s,t)}\bigg), \qquad s,t\in \mathbb D.
    \end{align}
\end{proposition}
Even in this simple Gaussian white-noise setting, the same issue as in 
Proposition~\ref{Prop1} arises. If weak convergence to a centered Gaussian 
process $\boldsymbol{\mathcal G}'=\{\mathcal G'(s)\}_{s\in\mathbb D}$ were 
to hold, it would be given by
\[
    \mathcal G'(s)
    \coloneqq
    \frac{2}{\pi}\bigg(\sqrt{2}\frac{\mathbb B(s)}{s}
    +
    \frac{1}{\sqrt{3s}}\,\mathbb U(-\log(s))\bigg),
    \qquad s\in[1/2,1],
\]
where $\mathbb B = \{\mathbb B(s)\}_{s\in\mathbb D}$ is a Brownian motion and 
$\mathbb U=\{\mathbb U(s)\}_{s\geq0}$ is an independent stationary 
Ornstein--Uhlenbeck process with $\Cov(\mathbb U(s),\mathbb U(t))=e^{-\frac 32|s-t|}.$ Since this limiting process is a sum of two independent components with distinct covariance structures, replacing the discrete statistic by its integral counterpart does not resolve the problem: the resulting self-normalized statistic does not have a pivotal limit of the form obtained in Theorem~\ref{Thm1}.

\begin{remark}
\begin{itemize}
    \item[\textnormal{(a)}]The arguments in Section \ref{Sec21}, where we considered the domain $\mathbb D=[1/2,1]$ for clarity, remain valid with $1/2$ replaced by any $\epsilon\in(0,1)$. In Section~\ref{Sec22}, however, the restriction $\mathbb D\subset[1/2,1]$ is required due to the more intricate covariance structure; see Step~2 of Section~\ref{Sec42} (on page~\pageref{eq47}).
    
    \item[\textnormal{(b)}]The innovation process $(\varepsilon_k)$ is assumed to be Gaussian white noise for simplicity. In this case, the associated periodogram ordinates $I_{N,\varepsilon}(\lambda_j)$ are independent across $j$ and exponentially distributed. Although these properties do not hold in general, our theory extends to centered i.i.d.\ innovations with a finite eighth moment, which is sufficient for the covariance calculations in Section~\ref{Sec22}. In this more general setting, the periodogram ordinates form an asymptotic white noise; see, e.g., \cite[Section 10]{BrockwellDavis1991}. For details on the corresponding covariance calculations, we refer to \citet[][Section A.4]{DetteKuehnert2026}, where distributional symmetry of the innovations is imposed to simplify the covariance calculations.
\end{itemize}
\end{remark}

\section{Discussion}\label{Sec3}

This article establishes weak limits for sequential estimators of spectral density integrals for real-valued linear processes. We consider integrals of the weighted and squared spectral density. For the former, under H\"older continuity of the weight function, we obtain pivotal weak limits free of unknown spectral quantities by using a sequential version of the weighted integrated periodogram. For the latter, we obtain a weak limit given by the sum of two Gaussian processes with distinct covariance structures, which precludes pivotal self-normalization as in the linear case. We further investigate whether this difficulty can be overcome by replacing the discrete estimator with a sequential version of the continuously integrated squared periodogram. For this integral counterpart, we derive the second-order structure, assuming that the process itself is a Gaussian white noise. We find that, even in this simple setting, the same structural difficulty persists, again precluding pivotal self-normalization.

While our analysis of the non-linear case focuses on the squared spectral density, the structure of the required estimators suggests that similar difficulties may arise for other non-linear functions $\phi$. A formal characterization of the class of functions $\phi$ for which $\Phi$ in \eqref{eq1} admits a pivotal weak limit remains an open problem. Extensions to the multivariate setting appear possible with additional technical effort.

\section{Proofs}\label{Sec4}

Throughout the proofs, we use the following notation and conventions. We write $\delta_{ij}$ for the Kronecker delta and $\mathds{1}_A$ for the indicator function of a set $A.$ Further, we use $\propto$ to denote proportionality, $\lesssim$ for inequalities up to a multiplicative constant independent of the variables involved, $\sim$ for asymptotic equivalence, and $\asymp$ for equivalence up to multiplicative constants. Unless stated otherwise, all limits are taken as $N\to\infty$. Functions defined on $[-\pi,\pi]$ are extended $2\pi$-periodically whenever necessary. Further, for brevity, we frequently use the notation
\[
    w_f(x) = w(x)f(x), \qquad f_{[1]}(\lambda_h) = f(\lambda_h)f(\lambda_{h-1}), \qquad g_1\cdots g_n(x) = g_1(x)\cdots g_n(x)\,.
\]

Throughout, $\mathbb D = [1/2,1].$ We define
\begin{align}\label{eq11}
    \widetilde I_N^s(\lambda)\coloneqq2\pi f(\lambda)I^s_{N,\varepsilon}(\lambda), \qquad \lambda \in [-\pi, \pi],~ s\in\mathbb D,
\end{align}
with $I^s_{N,\varepsilon}$ referring to the sequential periodogram in \eqref{eq3} with $X_k$ replaced by $\varepsilon_k$. The upper summation limit of $I^s_{N,\varepsilon}$ is frequently denoted $N_s = \lfloor sN\rfloor,$ and we set $N_{st} \coloneqq \lfloor \min(s,t)N\rfloor.$ We also introduce the (sequential) \emph{discrete Fourier transform} (DFT) as 
\begin{align}\label{eq12}
    J_N^s(\lambda)\coloneqq\frac{1}{\sqrt{2\pi sN}}\sum_{k=1}^{\lfloor sN\rfloor}X_k \e^{-\mathrm{i}k\lambda}, \qquad s \in \mathbb D, ~ \lambda \in [-\pi, \pi],
\end{align}
and write $J^s_{N,\varepsilon}(\lambda)$ if the $X_k's$ are replaced by the $\varepsilon_k's,$ and $J_N(\lambda)$ and $J_{N,\varepsilon}(\lambda)$ for the non-sequential DFTs (i.e.\ where $s=1$). Obviously, by definition,
\begin{align}
    I_N^s(\lambda) = \big|J_N^s(\lambda)\big|^2, \qquad 
    I^s_{N, \varepsilon}(\lambda) = \big|J^s_{N, \varepsilon}(\lambda)\big|^2, \qquad s \in \mathbb D, ~ \lambda \in [-\pi, \pi].
\end{align}
In addition, $\widetilde\Phi_N(s)$ and $\widetilde\Psi_N(s)$ are defined as $\widetilde\Phi_N(s)$ in \eqref{eq4} and $\widetilde\Psi_N(s)$ in \eqref{eq7}, respectively, with $I_N^s$ replaced by $\widetilde I_N^s$, and
\begin{alignat}{4}
    \Phi^{\circ}_N(s) 
    &\coloneqq \Phi_N(s)-\Phi, 
    \qquad && \widetilde\Phi^{\circ}_N(s) &&\coloneqq \widetilde\Phi_N(s)-\Phi,
    \qquad &&s\in\mathbb D,\label{eq13}\\
    \Psi^{\circ}_N(s) 
    &\coloneqq \Psi_N(s)-\Psi(s), 
    \qquad && \widetilde\Psi^{\circ}_N(s) &&\coloneqq \widetilde\Psi_N(s)-\Psi(s),
    \qquad &&s\in\mathbb D,\label{eq14}
\end{alignat}
with $\Phi$ and $\Psi(s)$ from \eqref{eq1} and \eqref{eq8}, respectively.

Our proofs also rely on the following kernels and related quantities; some of their properties are collected in Appendix~\ref{Sec5}. For $K\in\N$, define the full and upper (discrete) \emph{Dirichlet kernels} by
\begin{align}
    D_K(x) \coloneqq \sum_{j=-K}^K \e^{-\imag jx},
    \qquad
    D_K^+(x) \coloneqq \sum_{j=1}^K \e^{-\imag jx},
    \qquad x \in \R,
\end{align}
respectively, and the (discrete) \emph{Fejér kernel} by
\begin{align}
    F_K(x)
    \coloneqq \frac{1}{K}\sum_{j=0}^{K-1} D_j(x)
    = \sum_{|j|<K}\Big(1-\frac{|j|}{K}\Big)\e^{-\imag jx},
    \qquad x \in \R.
\end{align}
The Fejér kernel is symmetric and non-negative, with $F_K(0)=K$. Finally, for $g:[-\pi,\pi]\to\R$, define the (discrete) \emph{Fejér mean} of order $K\in\N$ by
\begin{align}
    \sigma_K(g,x)
    \coloneqq \frac{1}{N}\sum_{j=0}^{N-1}g(\lambda_j)F_K(x-\lambda_j),
    \qquad x \in \R.
\end{align}
We also use $n$th-order cumulants $\kappa(X_1,\dots,X_n)$ and write $\kappa_n(X)\coloneqq\kappa(X,\dots,X)$; see \citet[Section~2.3]{Brillinger2001} for their definition and properties.

\subsection{Proof of Theorem \ref{Thm1}}\label{Sec41}

The claim follows by establishing the postulated weak convergence in \eqref{eq5} in $\ell^\infty(\mathbb D)$ and verifying the proposed self-normalization. Related weak convergence results for processes with a similar triangular-array structure were obtained by \citet{ChenRomano1999}, but are not directly applicable here. We therefore consider the more tractable process $\widetilde{\boldsymbol{\mathcal F}}_N$, defined as $\boldsymbol{\mathcal F}_N$ in \eqref{eq5} with $\Phi_N$ replaced by $\widetilde{\Phi}_N$ from \eqref{eq13}. Weak convergence transfers by Slutsky's theorem to $\boldsymbol{\mathcal F}_N$ if
\[
    \big\|\widetilde{\boldsymbol{\mathcal F}}_N
    -\boldsymbol{\mathcal F}_N\big\|_\infty
    =\,
    \sup_{s\in\mathbb D}
    \big|\widetilde{\mathcal F}_N(s)-\mathcal F_N(s)\big|
    \,=\,
    o_{\Prob}(1),
\]
which follows from
\begin{align}\label{eq15}
\sup_{s\in\mathbb D}
\big|\widetilde{\Phi}_N(s)-\Phi_N(s)\big|
&=
o_{\Prob}\big(N^{-1/2}\big).
\end{align}

The proof is organized as follows. Step~1 establishes \eqref{eq15}. Steps~2--3 characterize the first- and second-order structure of $\widetilde{\Phi}_N$, providing the ingredients for Steps~4--5, which establish finite-dimensional (fidi) convergence and asymptotic tightness of $\widetilde{\boldsymbol{\mathcal F}}_N$, respectively. These results, together with \eqref{eq15} and Slutsky's theorem, yield the weak convergence in \eqref{eq5}. Step~6 then verifies the proposed self-normalization.

\paragraph{Step~1.}
The proof of \eqref{eq15} builds on arguments from \citet{DetteKinsvaterVetter2011} for the corresponding non-sequential setting $(s=1)$, as well as on ideas and representations from \citet[][Section~10]{BrockwellDavis1991}. In what follows, let
\begin{align}\label{eq16}
\Xi(\lambda)
\coloneqq
\sum^\infty_{\ell=-\infty}\psi_\ell\e^{-\imag\ell\lambda},
\qquad \lambda\in[-\pi,\pi].
\end{align}
Let $m\in\{\lfloor N/2\rfloor,\ldots,N\}$, as
$\lfloor N/2\rfloor\leq\lfloor Ns\rfloor\leq N$ for
$s\in\mathbb D=[1/2,1]$. For the (non-sequential) DFTs in
\eqref{eq12} (where $s=1$), the linear representation in \eqref{eq2}
yields
\begin{align}\label{eq17}
J_m(\lambda)
&=
\Xi(\lambda)J_{m,\varepsilon}(\lambda)
+Y_m(\lambda),
\qquad \lambda\in[-\pi,\pi],
\end{align}
where
\begin{alignat}{2}
    Y_m(\lambda)
    &\coloneqq
    \frac{1}{\sqrt{2\pi m}}
    \sum^\infty_{\ell=-\infty}
    \psi_\ell\e^{-\imag\ell\lambda}
    U_{m,\ell}(\lambda),
    &&\qquad \lambda\in[-\pi,\pi], \label{eq18}\\
    U_{m,\ell}(\lambda)
    &\coloneqq
    \sum_{r=1-\ell}^{m-\ell}
    \varepsilon_r\e^{-\imag r\lambda}
    -
    \sum_{r=1}^{m}
    \varepsilon_r\e^{-\imag r\lambda},
    &&\qquad \lambda\in[-\pi,\pi].
\end{alignat}
By the definition of $\widetilde I_m(\lambda)=\widetilde I_m^s(\lambda)$
in \eqref{eq11} and \eqref{eq17},
\begin{align}\label{eq19}
R_m(\lambda)
\coloneqq
I_m(\lambda)-\widetilde I_m(\lambda)
&=
2\Re\!\left(
\Xi(\lambda)J_{m,\varepsilon}(\lambda)
\overline{Y_m(\lambda)}
\right)
+
|Y_m(\lambda)|^2,
\end{align}
where $\Re(z)$ denotes the real part of $z\in\C$. We next derive
asymptotic upper bounds for these boundary terms uniformly over $m$.

\bigskip

\noindent{\bf Step~1.1.}
For every $\ell\in\mathbb Z$, $U_{m,\ell}$ admits the representation
\[
    U_{m,\ell}(\lambda)
    =
    \sum_{r\in B_{m,\ell}}
    c_{m,\ell,r}\varepsilon_r\e^{-\imag r\lambda},
\]
for a suitable index set $B_{m,\ell}\subset\mathbb Z$ and coefficients
$c_{m,\ell,r}\in\{-1,1\}$ satisfying
\[
    |B_{m,\ell}|
    \leq
    2\min(|\ell|,m).
\]
Since $(\varepsilon_k)$ is a Gaussian white noise, it follows that
$$
    \sup_{\lambda\in[-\pi,\pi]}
    \sup_{\lfloor N/2\rfloor\leq m\leq N}
    \Exp|U_{m,\ell}(\lambda)|^2
    \lesssim |\ell|.
$$
Hence, by the definition of $Y_m(\lambda)$ in \eqref{eq18} and
Minkowski's inequality,
\begin{align}
\sup_{\lambda\in[-\pi,\pi]}
\sup_{\lfloor N/2\rfloor\leq m\leq N}
\nu_2\big(Y_m(\lambda)\big)
&\lesssim
\bigg(\frac{1}{2\pi\lfloor N/2\rfloor}\bigg)^{1/2}
\sum^\infty_{\ell=-\infty}
|\ell|^{1/2}|\psi_\ell|
=
O(N^{-1/2}),
\end{align}
where $\nu_p(X)\coloneqq(\Exp|X|^p)^{1/p}$. The last equality follows
from the summability condition in \eqref{eq2}.

Furthermore, for every fixed $q\geq1$, since $U_{m,\ell}$ is a sum of independent, centered Gaussian white noise innovations, which in particular have finite $2q$th moments, the Marcinkiewicz--Zygmund inequality yields
\[
    \sup_{\lambda\in[-\pi,\pi]}
    \sup_{\lfloor N/2\rfloor\leq m\leq N}
    \nu_{2q}\big(U_{m,\ell}(\lambda)\big)
    \lesssim
    |\ell|^{1/2}.
\]
Hence, another application of Minkowski's inequality gives
\begin{align}\label{eq20}
    \sup_{\lambda\in[-\pi,\pi]}
    \sup_{\lfloor N/2\rfloor\leq m\leq N}
    \nu_{2q}\big(Y_m(\lambda)\big)
    =
    O(N^{-1/2}).
\end{align}

\medskip

\noindent{\bf Step~1.2.} We first derive an upper bound that will be useful later in the proof. Since the remainder term in \eqref{eq15} involves the frequencies $j=1,\ldots,\lfloor N/2\rfloor$, define
$$
    \xi_{N,r}
    \coloneqq
    \frac{1}{N}
    \sum_{j=1}^{\lfloor N/2\rfloor}
    w(\lambda_j)\Xi(\lambda_j)
    \e^{-\imag r\lambda_j},
    \qquad r\in\mathbb Z.
$$
By construction, $\xi_{N,r}$ is $N$-periodic in $r$. Moreover,
Parseval's identity for the discrete Fourier transform and the boundedness of $w$ and $\Xi$ give
\begin{align}\label{eq21}
    \sum_{r=0}^{N-1}|\xi_{N,r}|^2
    &=
    \frac{1}{N}
    \sum_{j=1}^{\lfloor N/2\rfloor}
    |w(\lambda_j)\Xi(\lambda_j)|^2
    = O(1),
\end{align}
uniformly in $N.$ For completeness, the boundedness of $\Xi$ follows directly from
$\sum_\ell|\psi_\ell|<\infty$, which is implied by the summability
condition in \eqref{eq2}. In addition, $\Xi$ is H\"older continuous of order $\alpha$. Indeed, since $\lambda\mapsto \e^{-\imag\lambda}$ is Lipschitz continuous and hence
H\"older continuous of order $\alpha\in(1/2,1]$, we have
\begin{align*}
    |\Xi(\lambda)-\Xi(\lambda')|
    &\lesssim
    \bigg(\sum_{\ell=-\infty}^{\infty}
    |\ell|^\alpha|\psi_\ell|\bigg)
    |\lambda-\lambda'|^\alpha,
\end{align*}
where the sum is finite by \eqref{eq2}. Combined with the assumed
H\"older continuity of $w$ of order $\beta\in(1/2,1]$, this shows that
$\lambda\mapsto w(\lambda)\Xi(\lambda)$ is H\"older continuous of order
$\gamma=\min(\alpha,\beta)\in(1/2,1]$.

\bigskip

\noindent{\bf Step~1.3.}
We next control the two remainder terms in \eqref{eq19} after
summation over the Fourier frequencies. Since $   \min(|\ell|,N) \leq N^{1-\alpha}|\ell|^\alpha$ for any $\ell \in \Z,$ the summability condition in \eqref{eq2} implies
\begin{align}\label{eq22}
\sum_{\ell=-\infty}^{\infty}
\min(|\ell|,N)|\psi_\ell|
&\leq
N^{1-\alpha}
\sum_{\ell=-\infty}^{\infty}
|\ell|^\alpha|\psi_\ell|
=
O(N^{1-\alpha})
=
o(N^{1/2}),
\end{align}
where the last relation follows from $\alpha\in(1/2,1]$.

We first consider the linear term in \eqref{eq19}. By the
representation of $U_{m,\ell}$ from Step~1.1 and the definitions of
$J_{m,\varepsilon}$ and $Y_m$, summation over the Fourier frequencies
$j=1,\ldots,\lfloor N/2\rfloor$ gives
\begin{align*}
\sum_{j=1}^{\lfloor N/2\rfloor}
w(\lambda_j)\Xi(\lambda_j)
J_{m,\varepsilon}(\lambda_j)
\overline{Y_m(\lambda_j)}
&=
\frac{N}{2\pi m}
\sum_{\ell=-\infty}^{\infty}
\psi_\ell
\sum_{r\in B_{m,\ell}}
c_{m,\ell,r}\varepsilon_r
\sum_{t=1}^{m}
\varepsilon_t
\xi_{N,t-\ell-r}.
\end{align*}
Since $m\leq N,$ as $\xi_{N,r}$ is $N$-periodic in $r$, and because of
\eqref{eq21}, 
\begin{align*}
\sum_{t=1}^{m}
|\xi_{N,t-\ell-r}|^2
\leq
\sum_{u=0}^{N-1}|\xi_{N,u}|^2
= O(1),
\end{align*}
uniformly in $N,m,\ell,$ and $r.$ Hence, for every $q\geq1$, the Marcinkiewicz--Zygmund inequality yields
\begin{align*}
    \sup_{\substack{\lfloor N/2\rfloor\leq m\leq N\\
                    \ell,r\in\mathbb Z}}\nu_{2q}\bigg(
    \sum_{t=1}^{m}
    \varepsilon_t\xi_{N,t-\ell-r}
    \bigg)
    \,\lesssim\,
    \sup_{\substack{\lfloor N/2\rfloor\leq m\leq N\\
                    \ell,r\in\mathbb Z}}\bigg(
    \sum_{t=1}^{m}
    |\xi_{N,t-\ell-r}|^2
    \bigg)^{1/2}
    \,=\,
    O(1)\,.
\end{align*}
Using Minkowski's inequality, followed by H\"older's inequality, we consequently obtain
\begin{align*}
    &\sup_{\lfloor N/2\rfloor\leq m\leq N}
    \nu_q\bigg(
    \sum_{j=1}^{\lfloor N/2\rfloor}
    w(\lambda_j)\Xi(\lambda_j)
    J_{m,\varepsilon}(\lambda_j)
    \overline{Y_m(\lambda_j)}
    \bigg)
    \\
    &\lesssim
    \frac{N}{m}
    \sum_{\ell=-\infty}^{\infty}
    |\psi_\ell|
    \sum_{r\in B_{m,\ell}}
    |c_{m,\ell,r}|
    \nu_{2q}(\varepsilon_r)
    \nu_{2q}\bigg(
    \sum_{t=1}^{m}
    \varepsilon_t\xi_{N,t-\ell-r}
    \bigg)
    \\
    &\lesssim
    N^{1-\alpha},
\end{align*}
where we used $m\asymp N$, $|B_{m,\ell}|\leq2\min(|\ell|,m)$, and
\eqref{eq22}.

Next, choose $q$ sufficiently large such that $1/q < \alpha-1/2.$ Since there are at most $N$ possible values of $m$, the elementary
inequality
$$
    \max_m|Z_m|
    \leq
    \bigg(\sum_m|Z_m|^q\bigg)^{1/q}
$$
gives
\begin{align}\label{eq23}
&\nu_q\Bigg(
\max_{\lfloor N/2\rfloor\leq m\leq N}
\Bigg|
\sum_{j=1}^{\lfloor N/2\rfloor}
w(\lambda_j)\Xi(\lambda_j)
J_{m,\varepsilon}(\lambda_j)
\overline{Y_m(\lambda_j)}
\Bigg|\,
\Bigg)
\nonumber\\
&\leq
N^{1/q}
\sup_{\lfloor N/2\rfloor\leq m\leq N}
\nu_q\Bigg(
\sum_{j=1}^{\lfloor N/2\rfloor}
w(\lambda_j)\Xi(\lambda_j)
J_{m,\varepsilon}(\lambda_j)
\overline{Y_m(\lambda_j)}
\Bigg)
\notag\\[1ex]
&=
O(N^{1-\alpha+1/q})
=
o(N^{1/2}).
\end{align}

It remains to control the quadratic term in \eqref{eq19}.
By \eqref{eq20}, for every fixed $q\geq1$,

$$
    \sup_{\lambda\in[-\pi,\pi]}
    \sup_{\lfloor N/2\rfloor\leq m\leq N}
    \nu_{2q}\big(Y_m(\lambda)\big)
    =
    O(N^{-1/2}).
$$
Hence, boundedness of $w$ and Minkowski's inequality yield
\begin{align*}
\sup_{\lfloor N/2\rfloor\leq m\leq N}
\nu_q\Bigg(
\sum_{j=1}^{\lfloor N/2\rfloor}
w(\lambda_j)|Y_m(\lambda_j)|^2
\Bigg)
&\lesssim
\sum_{j=1}^{\lfloor N/2\rfloor}
\sup_{\lfloor N/2\rfloor\leq m\leq N}
\nu_{2q}\big(Y_m(\lambda_j)\big)^2
=
O(1).
\end{align*}
Choosing $q>2$ and arguing as above yields
\begin{align}
\nu_q\Bigg(
\max_{\lfloor N/2\rfloor\leq m\leq N}
\Bigg|
\sum_{j=1}^{\lfloor N/2\rfloor}
w(\lambda_j)|Y_m(\lambda_j)|^2
\Bigg|
\Bigg)
\nonumber
&\leq
N^{1/q}
\sup_{\lfloor N/2\rfloor\leq m\leq N}
\nu_q\Bigg(
\sum_{j=1}^{\lfloor N/2\rfloor}
w(\lambda_j)|Y_m(\lambda_j)|^2
\Bigg)\notag\\[1ex]
&=
O(N^{1/q})
=
o(N^{1/2}).
\label{eq24}
\end{align}

Together with \eqref{eq19}, \eqref{eq23}, and
\eqref{eq24}, and using Lyapunov's inequality,  we thus obtain
$$
    \Exp\!\Bigg(
        \max_{\lfloor N/2\rfloor\leq m\leq N}
        \Bigg|
            \sum_{j=1}^{\lfloor N/2\rfloor}
            w(\lambda_j)R_m(\lambda_j)
        \Bigg|
    \Bigg)
    =
    o(N^{1/2}).
$$
Hence, by Markov's inequality,
\begin{align}\label{eq25}
\max_{\lfloor N/2\rfloor\leq m\leq N}
\Bigg|
\sum_{j=1}^{\lfloor N/2\rfloor}
w(\lambda_j)R_m(\lambda_j)
\Bigg|
=
o_{\Prob}(N^{1/2}).
\end{align}

Finally, let $N_s=\lfloor sN\rfloor$. By the definitions of
$I_N^s$ and $\widetilde I_N^s$, and since $N_s/(sN)\leq1$ for
$s\in\mathbb D=[1/2,1]$, \eqref{eq25} implies
\begin{align*}
    \sup_{s\in\mathbb D}
    |\Phi_N(s)-\widetilde\Phi_N(s)|
    &=
    \frac{4\pi}{N}
    \sup_{s\in\mathbb D}
    \Bigg|
    \frac{N_s}{sN}
    \sum_{j=1}^{\lfloor N/2\rfloor}
    w(\lambda_j)R_{N_s}(\lambda_j)
    \Bigg|
\\[1ex]
&=
o_{\Prob}(N^{-1/2}),
\end{align*}
which proves \eqref{eq15}.

\paragraph{Step~2.}
Here, we show
\begin{align}\label{eq26}
    \sup_{s\in\mathbb D}\,\big|\Exp(\widetilde\Phi^{\circ}_N(s))\big| = o(N^{-1/2}),
\end{align}
with $\widetilde\Phi^{\circ}_N(s) = \widetilde\Phi_N(s) - \Phi.$ As $(\varepsilon_k)$ is a Gaussian white noise with unit variance, $\Exp(\varepsilon_m\varepsilon_n) = \delta_{mn}$ (the Kronecker delta). Thus, for the sequential periodogram $I^s_{N,\varepsilon}$ in \eqref{eq3} (with $X_k$ replaced by $\varepsilon_k$), we obtain
\begin{align*}
    \Exp\!\big(I^s_{N,\varepsilon}(\lambda_j)\big) \,=\, 
    \frac{1}{2\pi sN}\sum_{m,n=1}^{N_s}\Exp(\varepsilon_m\varepsilon_n)\e^{-\imag(m-n)\lambda_j}
    \,=\, \frac{N_s}{2\pi sN}\,, \qquad 1 \leq h,j \leq \lfloor N/2\rfloor.
\end{align*}
Further, as the weight function $w$ and the spectral density $f$ are symmetric and H\"older continuous of orders $\alpha, \beta,$ respectively, $x \mapsto w_f(x) = w(x)f(x)$ is symmetric and H\"older continuous of order $\gamma = \min(\alpha, \beta).$ Hence, as $\lambda_j - \lambda_{j-1} = 2\pi/N,$ and since $N_s = \lfloor sN\rfloor \asymp N,$ where $s \in [1/2,1],$
\begin{align}
    \big|\Exp(\widetilde\Phi^{\circ}_N(s))\big|
    &= \Bigg|\,\frac{8\pi^2}{N}\sum_{j=1}^{\lfloor N/2\rfloor}
    w_f(\lambda_j)
    \Exp\!\big(I^s_{N,\varepsilon}(\lambda_j)\big)
    - \int^\pi_{-\pi}w_f(\lambda)\,\mathrm{d}\lambda\,\Bigg|
    \notag\\
    &= 2\,\Bigg|\,
    \frac{2\pi N_s}{sN^2}\sum_{j=1}^{\lfloor N/2\rfloor}
    w_f(\lambda_j)
    - \int_0^\pi w_f(\lambda)\,\mathrm{d}\lambda\,\Bigg|\notag\allowdisplaybreaks\\
    &\leq 2\sum_{j=1}^{\lfloor N/2\rfloor}\int^{\lambda_j}_{\lambda_{j-1}}
    \underbrace{\,\big|w_f(\lambda_j)
    -  w_f(\lambda)\big|\,}_{\lesssim\;|\lambda_j - \lambda|^\gamma \;\leq \;(2\pi/N)^\gamma}\,\mathrm{d}\lambda  + \underbrace{\,\bigg|\frac{\lfloor sN\rfloor}{sN} -1 \bigg|\,\Bigg|\,\frac{4\pi}{N}\sum_{j=1}^{\lfloor N/2\rfloor}
    w_f(\lambda_j)\,\Bigg|\,}_{=\,O(N^{-1})\,}\notag\\
    &= O(N^{-\gamma})\,,\notag
\end{align}
uniformly in $s.$ Hence, since $\alpha, \beta \in (1/2,1],$ and thus $\gamma = \min(\alpha, \beta) \in (1/2,1],$ \eqref{eq26} follows.

\paragraph{Step~3.} Next, we determine the limiting covariances
\begin{align}
    \lim_{N\to \infty}N\Cov\big(\widetilde \Phi^{\circ}_N(s), \widetilde \Phi^{\circ}_N(t)\big), \qquad s,t \in \mathbb D.
\end{align}
Recall that $N_s = \lfloor sN\rfloor$ and $N_{st} = \lfloor \min(s,t)N\rfloor.$ Since $\Cov(\varepsilon_m\varepsilon_n,\varepsilon_o\varepsilon_p) = \delta_{mp}\delta_{no} + \delta_{mo}\delta_{np}$ by Isserlis's theorem,
\begin{align}
    \Cov\!\big(I^s_{N,\varepsilon}(\lambda_h), I^t_{N,\varepsilon}(\lambda_j)\big) 
    &= \frac{1}{4\pi^2stN^2}\sum_{m,n=1}^{N_s}\sum_{o,p=1}^{N_t}\Cov(\varepsilon_m\varepsilon_n,\varepsilon_o\varepsilon_p)\e^{-\imag[(m-n)\lambda_h+(o-p)\lambda_j]}\notag\\
    &= \frac{1}{4\pi^2stN^2}\bigg(\sum_{m,n=1}^{N_{st}}\Big[\e^{-\imag(m-n)\lambda_{h-j}} + \e^{-\imag(m-n)\lambda_{h+j}}\Big]\bigg)\notag\\
    &= \frac{N_{st}}{4\pi^2stN^2}\big[
    F_{N_{st}}(\lambda_{h+j}) + F_{N_{st}}(\lambda_{h-j})
    \big],\notag
\end{align}
where $F_{N_{st}}$ is the Fejér kernel of order $N_{st}$ and
$\sigma_{N_{st}}$ the associated Fejér mean we use in the following; see the beginning of this section and Appendix \ref{Sec5}. By symmetry and
$2\pi$-periodicity of $w, f$, and the Fejér kernel, with $w_f(x) = w(x)f(x),$ we obtain for odd $N$ (the even case follows analogously)
\begin{align*}
    &4\sum_{h,j=1}^{\lfloor N/2 \rfloor}
    w_f(\lambda_h)w_f(\lambda_j)
    \big[
        F_{N_{st}}(\lambda_{h+j})
        +
        F_{N_{st}}(\lambda_{h-j})
    \big]\\
    &=
    \sum_{h,j=0}^{N-1}
    w_f(\lambda_h)w_f(\lambda_j)
    \big[
        F_{N_{st}}(\lambda_{h+j})
        +
        F_{N_{st}}(\lambda_{h-j})
    \big]
    -4w_f(0)
    \sum_{h=0}^{N-1}
    w_f(\lambda_h)F_{N_{st}}(\lambda_h)
    +2N_{st}w^2_f(0)
    \\
    &=
    2N\sum_{h=0}^{N-1}
    w_f(\lambda_h)
    \sigma_{N_{st}}(w_f,\lambda_h)
    -4Nw_f(0)\sigma_{N_{st}}(w_f,0)
    +2N_{st}w^2_f(0),
\end{align*}
where the last two summands are $O(N)$ due to boundedness of $w_f$ and Lemma \ref{Lem3}. Hence,
\begin{align}
    N\Cov\!\big(\widetilde \Phi^{\circ}_N(s), \widetilde \Phi^{\circ}_N(t)\big)
    &= \frac{64\pi^4}{N}\sum_{h,j=1}^{\lfloor N/2\rfloor}w_f(\lambda_h)w_f(\lambda_j)\Cov\!\big(I^s_{N,\varepsilon}(\lambda_h), I^t_{N,\varepsilon}(\lambda_j)\big)\notag\allowdisplaybreaks\\
    &\sim \frac{8\pi^2}{\max(s,t)N}\sum_{h=0}^{N-1}w_f(\lambda_h)\sigma_{N_{st}}(w_f,\lambda_h)\notag\allowdisplaybreaks\\
    &\!\stackrel{N \to \infty}{\longrightarrow}\, \frac{4\pi}{\max(s,t)}\int^\pi_{-\pi}w^2_f(\lambda)\,\mathrm{d}\lambda
    \,=\,
    \tau^2_{w_f}\,\frac{1}{\max(s,t)}\,,
\end{align}
where we used that $w_f$ is H\"older continuous of order $\gamma \in(1/2,1]$, together with the uniform convergence of the Fejér mean $\sigma_{N_{st}}$ established in Lemma~\ref{Lem3}, noting that $N_{st}\asymp N$ for $s,t\in\mathbb D$.

\paragraph{Step~4.} Here, we prove fidi convergence of \eqref{eq5}. As the limiting process $\boldsymbol{\mathcal F}$ therein is claimed to be centered Gaussian, we show for arbitrary $n\in\N$ and $s_1,\dots,s_n \in \mathbb{D} =[1/2, 1],$
\begin{align}
    \sqrt{N}\bigl(\widetilde \Phi^{\circ}_N(s_1), \dots, \widetilde \Phi^{\circ}_N(s_n)\bigr) 
    \,\stackrel{d}{\longrightarrow}\, \mathcal{N}\bigl(0,\Sigma_\Phi(s_1,\dots,s_n)\bigr),
\end{align}
for some non-negative matrix $\Sigma_\Phi(s_1,\dots,s_n)$. By the Cramér--Wold device, it suffices to show that every linear combination of the pre-limit is asymptotically centered and Gaussian. This is given if the first moment converges to zero (see Step 2), the limiting covariance is non-degenerate (see Step 3), and all cumulants of order $k\geq3$ vanish asymptotically. It remains to verify the latter property. For clarity, we provide the argument for $k=3.$ Fix any $t_1,t_2,t_3\in\mathbb D$ and, without loss of generality, assume that $t_1\geq t_2\geq t_3$. Indeed, due to multilinearity and translation invariance of cumulants, boundedness of $w_f,$ and since $N_{t_1} \geq N_{t_2} \geq N_{t_3},$ where $N_{t_i} = \lfloor t_iN\rfloor,$ by the inequality $|D^+_K(x)| \leq K$ and Lemma \ref{Lem3}, we obtain
\begin{align} 
    &\Big|\kappa\Big(\sqrt{N}\widetilde \Phi^{\circ}_N(t_1), \sqrt{N}\widetilde \Phi^{\circ}_N(t_2), \sqrt{N}\widetilde \Phi^{\circ}_N(t_3)\Big)\Big|\notag\\[1ex]
    &= N^{-3/2}\big|\kappa\big(\widetilde \Phi_N(t_1), \widetilde \Phi_N(t_2), \widetilde \Phi_N(t_3)\big)\big|\notag\\
    &\propto N^{-3/2}\bigg|
    \sum_{j,k,\ell=1}^{\lfloor N/2\rfloor}
    w_f(\lambda_j)w_f(\lambda_k)w_f(\lambda_\ell)\,
    \kappa\big(
    I^{t_1}_{N,\varepsilon}(\lambda_j),
    I^{t_2}_{N,\varepsilon}(\lambda_k),
    I^{t_3}_{N,\varepsilon}(\lambda_\ell)
    \big)
    \bigg|
    \notag\allowdisplaybreaks\\
    &\lesssim N^{-9/2}\sum_{j, k, \ell=1}^{\lfloor N/2\rfloor}\,\bigg|\sum_{m, n=1}^{N_{t_1}}\sum_{o,p=1}^{N_{t_2}}\sum_{q,r=1}^{N_{t_3}}\e^{-\imag[(m - n)\lambda_{j} + (o - p)\lambda_{k} + (q - r)\lambda_{\ell}]}\kappa(\varepsilon_m\varepsilon_n, \varepsilon_o\varepsilon_p, \varepsilon_q\varepsilon_r)\bigg|\notag\allowdisplaybreaks\\
    &\lesssim N^{-9/2}\bigg(
    \sum_{j,k,\ell=1}^{\lfloor N/2\rfloor}
    \big|
    D^+_{N_{t_2}}(\lambda_{j+k})D^+_{N_{t_3}}(\lambda_{\ell-j})D^+_{N_{t_3}}(\lambda_{-k-\ell})
    \big|
    \,+\, \cdots
    \bigg)\notag \allowdisplaybreaks\\
    &\leq N^{-9/2}\bigg(
    N\sum_{j=1}^{\lfloor N/2\rfloor}
    \big|D^+_{N_{t_2}}(\lambda_j)\big|
    \sum_{k,\ell=1}^{\lfloor N/2\rfloor}
    \big|D^+_{N_{t_3}}(\lambda_{-k-\ell})\big|
    \,+\, \cdots
    \bigg)\notag \allowdisplaybreaks\\
    &= N^{-9/2}\Big(O\big(N^4\log^2(N)\big) \,+\, \cdots \Big)\notag\\
    &= o(1).\notag
\end{align}
The higher-order cumulants are treated analogously: for every fixed $k>3$, the connected Gaussian pairings yield the same type of products of upper Dirichlet kernels and, by Lemma \ref{Lem1}, an $o(1)$ bound after normalization.

\paragraph{Step~5.} Herein, we establish asymptotic tightness via moment bounds for increments  \citep[][Example 2.2.7]{VaartWellner2023}. Recall the notation $N_s = \lfloor sN\rfloor$ and $N_{st} = \lfloor \min(s,t)N\rfloor,$ where $s,t \in \mathbb D = [1/2,1].$ Throughout, let $s\geq t$ without loss of generality. Thus,
\begin{align*}
    N \geq N_s \geq N_t \geq N_{st},
\end{align*}
where these quantities are by their definition all asymptotically equivalent up to a factor, so
\begin{align*}
    N 
    \asymp N_s \asymp N_t \asymp N_{st}.
\end{align*}

Now, we derive moment bounds of the increments of our process. Since
$\widetilde I_N^s(\lambda_j) = 2\pi f(\lambda_j)I^s_{N,\varepsilon}(\lambda_j),$
\begin{align}
    N\Exp\bigl|\widetilde\Phi^{\circ}_N(s)-\widetilde\Phi^{\circ}_N(t)\bigr|^2
    &= \frac{64\pi^4}{N}\Exp \!\Bigg(\sum_{j=1}^{\lfloor N/2\rfloor}w_f(\lambda_j)\Big[I^s_{N,\varepsilon}(\lambda_j) - I^t_{N,\varepsilon}(\lambda_j) - \Exp\!\big(I^s_{N,\varepsilon}(\lambda_j) - I^t_{N,\varepsilon}(\lambda_j)\big)\Big]\Bigg)^{\!2}\notag\\
    &= \frac{4\pi^2}{N}\sum_{h,j=1}^{\lfloor N/2\rfloor}w_f(\lambda_h)w_f(\lambda_j)\,\kappa\Big(I^s_{N,\varepsilon}(\lambda_h) - I^t_{N,\varepsilon}(\lambda_h), I^s_{N,\varepsilon}(\lambda_j) - I^t_{N,\varepsilon}(\lambda_j)\Big).
\end{align}
Since $N_s \geq N_t,$ we have  $\kappa(I^s_{N,\varepsilon}(x), I^t_{N,\varepsilon}(y)) = \tfrac{t}{s}\kappa(I^t_{N,\varepsilon}(x), I^t_{N,\varepsilon}(y))$ by Isserlis's theorem. Thus, together with $1 + \frac{s(s-2t)}{t^2}= \frac{(s-t)^2}{t^2},$ we obtain
\begin{align*}
    &\kappa\Big(I^s_{N,\varepsilon}(\lambda_h) - I^t_{N,\varepsilon}(\lambda_h), I^s_{N,\varepsilon}(\lambda_j) - I^t_{N,\varepsilon}(\lambda_j)\Big)\notag\\
    &=
    \frac{1}{4\pi^2s^2N^2}\bigg(
    N_s\big[F_{N_s}(\lambda_{h+j}) + F_{N_s}(\lambda_{h-j})\big]   
    +
    \frac{s(s-2t)}{t^2}N_t\big[F_{N_t}(\lambda_{h+j}) + F_{N_t}(\lambda_{h-j})\big]
    \bigg)\notag\allowdisplaybreaks\\
    &=
    \frac{1}{4\pi^2s^2N^2}\bigg(
    (N_s\!-\!N_t)\big[F_{N_s}(\lambda_{h+j}) + F_{N_s}(\lambda_{h-j})\big] + N_t\Big[F_{N_s}(\lambda_{h+j}) + F_{N_s}(\lambda_{h-j}) - F_{N_t}(\lambda_{h+j}) - F_{N_t}(\lambda_{h-j})\Big]\notag\\
    &\qquad\qquad\quad + \frac{(s-t)^2}{t^2}N_t\big[F_{N_t}(\lambda_{h+j}) + F_{N_t}(\lambda_{h-j})\big]
    \bigg).
\end{align*}
Subsequently, we proceed as in Step~3 and express the corresponding sums
in terms of the Fejér means. Since
$1\geq s \geq t>\tfrac12$, we have
\[
    N_s\!-\!N_t\leq 1+|s-t|N,
    \qquad
    (s-t)^2\leq |s-t|\leq 1.
\]
Further, as $w_f$ is H\"older continuous, Lemma~\ref{Lem3} implies that $\sigma_{N_s}(w_f,\lambda_h)$ and $\sigma_{N_t}(w_f,\lambda_h)$ are uniformly bounded in $h.$ Hence, together with the arguments from Step~3 for the difference of the Fejér means, it follows for odd $N$ (the even case follows analogously) that
\begin{align*}
    &N\Exp\bigl|
    \widetilde\Phi^{\circ}_N(s)-\widetilde\Phi^{\circ}_N(t)
    \bigr|^2
    \\
    &=
    \frac{8\pi^2}{s^2N^2}
    \sum_{h=0}^{N-1}
    w_f(\lambda_h)
    \bigg(
        (N_s\!-\!N_t)\sigma_{N_s}(w_f,\lambda_h) + N_t\Big[
            \sigma_{N_s}(w_f,\lambda_h)
            -\sigma_{N_t}(w_f,\lambda_h)
            +\frac{(s-t)^2}{t^2}
             \sigma_{N_t}(w_f,\lambda_h)
        \Big]
    \bigg)
    \\
    &\quad\;\,
    -\frac{16\pi^2}{s^2N^2}w_f(0)
    \bigg(
        (N_s\!-\!N_t)\sigma_{N_s}(w_f,0)
        +N_t\Big[
            \sigma_{N_s}(w_f,0)
            -\sigma_{N_t}(w_f,0)
            +\frac{(s-t)^2}{t^2}
             \sigma_{N_t}(w_f,0)
        \Big]
    \bigg)
    \\
    &\quad\;\,
    +\frac{8\pi^2}{s^2N^3}w^2f^2(0)
    \bigg(
        (N_s\!-\!N_t)(N_s\!+\!N_t)
        +\frac{(s-t)^2}{t^2}N_t^2
    \bigg)
    \allowdisplaybreaks\\
    &\lesssim
    \frac{N_s\!-\!N_t}{N}
    +
    \frac{1}{N}
    \bigg|
    \sum_{h=0}^{N-1}
    w_f(\lambda_h)
    \Big(
        \sigma_{N_s}(w_f,\lambda_h)
        -\sigma_{N_t}(w_f,\lambda_h)
    \Big)
    \bigg|
    +(s-t)^2+\frac{1}{N}
    \\
    &\lesssim
    |s-t|+(s-t)^2+\frac{1}{N}\notag\\
    &\lesssim
    |s-t|+\frac{1}{N}.
\end{align*}

With $\|\cdot\|_2$ denoting the $L^2$-norm, we thus have
\begin{align*}
     \sqrt{N}\bigl\|\widetilde\Phi^{\circ}_N(s) - \widetilde\Phi^{\circ}_N(t)\bigr\|_2 
     \,\lesssim\, |s-t|^{1/2} + N^{-1/2} 
     \,\lesssim\, d(s,t),
\end{align*}
for all $s,t$, provided $2d(s,t) \ge \bar\eta \coloneqq N^{-1/2}$, with $d(s,t) \coloneqq |s-t|^{1/2}$. Hence, the conditions of \citet[][Lemma A.1]{KleyVolgushevDetteHallin2016} are met. With $\psi(x)=x^4$ and packing numbers $\mathcal D(\epsilon,d)\leq \lfloor \epsilon^{-2}\rfloor+1$ (the maximal number of points in $\mathbb{D}$ with pairwise distance exceeding $\epsilon$), it follows that for any $\delta>0$ and $\eta \geq \bar\eta$, there exists a random variable $S$ and a constant $K<\infty$ such that
\begin{align}\label{eq27}
    \sqrt{N}\!\sup_{d(s,t) \leq \delta}\,\bigl|\widetilde\Phi^{\circ}_N(s)-\widetilde\Phi^{\circ}_N(t)\bigr| \,\leq\, S + 2\sqrt{N}\!\sup_{d(s,t) \leq \bar\eta, t\in\tilde T}\,\bigl|\widetilde\Phi^{\circ}_N(s)-\widetilde\Phi^{\circ}_N(t)\bigr|\,,
\end{align}
where $\tilde T = \{t_1, \dots, t_q\}\subset \mathbb{D}$ denotes a set that contains at most $\mathcal D(\bar\eta, d)$ points, and
\begin{align*}
    \Prob(|S| > x) \leq\Bigg\{\psi\Bigg(x\bigg[8K\bigg(\int^\eta_{\bar\eta/2}\psi^{-1}(\mathcal D(\epsilon, d))\,\mathrm{d}\epsilon + (\delta +2\bar\eta)\psi^{-1}(\mathcal D^2(\eta, d))\bigg)\bigg]^{-1}\Bigg)\!\Bigg\}^{\!-1}\!, ~~ x>0.
\end{align*}
Thereby, due to 
\[
    \int^\eta_{\bar\eta/2}\psi^{-1}(\mathcal D(\epsilon, d))\,\mathrm{d}\epsilon
    \,\lesssim\, \int^\eta_0 \,\epsilon^{-1/2}\,\mathrm{d}\epsilon ~<~ \infty, \quad \eta \geq \bar\eta/2 > 0,
\]
and $\bar\eta = N^{-1/2}$ and $\psi(x) = x^4,$ we have
\begin{align}\label{eq28}
    \lim_{\delta \downarrow 0}\lim_{N\to\infty}\Prob(|S| > x) \leq \left(\frac{8K}{x}\int^\eta_0\,\epsilon^{-1/2} \,\mathrm{d}\epsilon\right)^{\!4}\!, \quad \eta, x >0.
\end{align}

\noindent Next, observe the second term in \eqref{eq27}. Since $\bar\eta = 1/N,$ the set $\tilde T = \{t_1, \dots, t_q\}$ in \eqref{eq27} contains up to $\mathcal{D}(\bar\eta, q) \leq N^2$ points. Therefore, as the process under consideration is a partial sum process and thus constant on the intervals $[\tfrac{\zeta-1}{N}, \tfrac{\zeta}{N})$ for $\lfloor N/2\rfloor < \zeta \leq N,$
\begin{align}
    \sup_{d(s,t) \leq \bar\eta, t\in\tilde T}\,\bigl|\widetilde\Phi^{\circ}_N(s) \!-\! \widetilde\Phi^{\circ}_N(t)\bigr| &= \max_{1\leq \zeta \leq q}\!\max\!\Big\{\bigl|\widetilde\Phi^{\circ}_N(t_\zeta \!+\! \tfrac 1 N)\! - \!\widetilde\Phi^{\circ}_N(t_\zeta)\bigr|\, , \bigl|\widetilde\Phi^{\circ}_N(t_\zeta)\! - \!\widetilde\Phi^{\circ}_N(t_\zeta \!-\! \tfrac 1 N)\bigr|\Big\}\notag\\
    &\leq \max_{\lfloor N/2\rfloor< \zeta \leq N}\,\bigl|\widetilde\Phi^{\circ}_N(\tfrac \zeta N) - \widetilde\Phi^{\circ}_N(\tfrac {\zeta-1} N)\bigr|\,.
\end{align}
Furthermore, \citep[cf.][p.~145]{vanderVaart1998},
\begin{align*}
    \Bigl\|\,\max_{\lfloor N/2\rfloor< \zeta \leq N}\,\bigl|\widetilde\Phi^{\circ}_N(\tfrac \zeta N) - \widetilde\Phi^{\circ}_N(\tfrac {\zeta-1} N)\bigr|\,\Bigr\|_4
    \leq
    N^{1/4}\!\max_{\lfloor N/2\rfloor< \zeta \leq N}
    \big\|\widetilde\Phi^{\circ}_N(\tfrac \zeta N) - \widetilde\Phi^{\circ}_N(\tfrac {\zeta-1} N)\big\|_4.
\end{align*}
Since each increment
\[
    \widetilde\Phi^{\circ}_N(\tfrac \zeta N)
    -
    \widetilde\Phi^{\circ}_N(\tfrac {\zeta-1} N)
\]
Since it is a polynomial of degree at most two in the Gaussian innovations, Gaussian hypercontractivity yields that its $L^4$-norm is bounded by a constant multiple of its $L^2$-norm, uniformly in $N$. Applying the second-moment bound above with $s=\zeta/N$ and $t=(\zeta-1)/N$ therefore gives
\[
    \max_{\lfloor N/2\rfloor< \zeta \leq N}
    \big\|\widetilde\Phi^{\circ}_N(\tfrac \zeta N) - \widetilde\Phi^{\circ}_N(\tfrac {\zeta-1} N)\big\|_4
    =
    O(N^{-1}).
\]
Consequently,
\begin{align*}
    \Big\|\sqrt{N}\!\sup_{d(s,t) \leq \bar\eta,\, t\in\tilde T}
    \bigl|\widetilde\Phi^{\circ}_N(s)-\widetilde\Phi^{\circ}_N(t)\bigr|\Big\|_4
    &\leq
    N^{3/4}\!\max_{\lfloor N/2\rfloor< \zeta \leq N}
    \big\|\widetilde\Phi^{\circ}_N(\tfrac \zeta N) - \widetilde\Phi^{\circ}_N(\tfrac {\zeta-1} N)\big\|_4
    \\
    &=O(N^{-1/4}).
\end{align*}

Finally, \eqref{eq27}--\eqref{eq28} and basic inequalities imply for any $\varepsilon > 0,$
\begin{align}
     &\lim_{\delta \downarrow 0}\lim_{N\to\infty}\Prob\Big(\sqrt{N}\!\sup_{d(s,t) \leq \delta}\,\bigl|\widetilde\Phi^{\circ}_N(s)-\widetilde\Phi^{\circ}_N(t)\bigr| > \varepsilon\Big)\notag\\
     &\leq\lim_{\delta \downarrow 0}\lim_{N\to\infty}\Prob(|S| > \varepsilon/2) \,+\, \lim_{\delta \downarrow 0}\lim_{N\to\infty}\Prob\Big(\sqrt{N}\!\sup_{d(s,t) \leq \bar\eta, t\in\tilde T}\,\bigl|\widetilde\Phi^{\circ}_N(s)-\widetilde\Phi^{\circ}_N(t)\bigr| > \varepsilon/4\Big)\notag\allowdisplaybreaks\\
     &\leq \lim_{\delta \downarrow 0}\lim_{N\to\infty}\Prob(|S| > \varepsilon/2) \,+\, \Big(\frac{4}{\varepsilon}\Big)^4\lim_{\delta \downarrow 0}\lim_{N\to\infty}\Big\|\sqrt{N}\!\sup_{d(s,t) \leq \bar\eta, t\in\tilde T}\,\bigl|\widetilde\Phi^{\circ}_N(s)-\widetilde\Phi^{\circ}_N(t)\bigr|\,\Big\|^4_4 \notag\\
     &= 0.\notag
\end{align}
Asymptotic tightness is thus established.

\paragraph{Step~6.} Finally, we establish the claimed pivotal weak limit by combining the preceding steps and applying self-normalization. We first summarize what has been established so far. Recall that $\boldsymbol{\mathcal{F}}_N = \{\mathcal{F}_N(s)\}_{s\in\mathbb D}$ and $   \widetilde{\boldsymbol{\mathcal F}}_N = \{\widetilde{\mathcal F}_N(s)\}_{s\in\mathbb D},$
where
\[
    \mathcal{F}_N(s)
    = \sqrt{N}\big(\Phi_N(s)-\Phi\big),
    \qquad
    \widetilde{\mathcal F}_N(s)
    = \sqrt{N}\big(\widetilde\Phi_N(s)-\Phi\big).
\]
Here, $\Phi$ denotes the integral in \eqref{eq1}, $\Phi_N(s)$ is the 
weighted integrated sequential periodogram defined in \eqref{eq4}, and 
$\widetilde\Phi_N(s)$ denotes its corresponding version obtained by 
replacing the sequential periodogram $I_N^s(\lambda)$ in \eqref{eq3} 
with $\widetilde I_N^s(\lambda)$ from \eqref{eq11}. Step~1 showed that the replacement of $I_N^s$ by $\widetilde I_N^s$ is
asymptotically negligible, that is,
\[
    \big\|
        \boldsymbol{\mathcal F}_N
        - \widetilde{\boldsymbol{\mathcal F}}_N
    \big\|_\infty
    = o_{\mathbb P}(1).
\]
Steps~2--5, in turn, established the weak convergence of the approximating
process $\widetilde{\boldsymbol{\mathcal F}}_N$ to the claimed Gaussian
limit. Combining these results and applying Slutsky's theorem therefore
yields
\[
    \boldsymbol{\mathcal F}_N
    =
    \widetilde{\boldsymbol{\mathcal F}}_N
    +
    \big(
        \boldsymbol{\mathcal F}_N
        - \widetilde{\boldsymbol{\mathcal F}}_N
    \big)
    \stackrel{d}{\longrightarrow}
    \tau_{w_f}
    \Big\{
        \frac{1}{s}\mathbb B(s)
    \Big\}_{s\in\mathbb D}\,,
\]
in $\ell^\infty(\mathbb D)$, where $\tau_{w_f}$ is given in \eqref{eq6}
and $\mathbb B=\{\mathbb B(s)\}_{s\in\mathbb D}$ is a Brownian motion.

This weak convergence provides the basis for the final self-normalization step. Since $w$ is non-zero (Lebesgue-)almost everywhere and $f\not\equiv 0$ by assumption, we have $\tau_{w_f}>0$. Applying the continuous mapping theorem to the map
\[
    \ell^\infty(\mathbb D)\ni x
    \mapsto
    \frac{x(1)}
    {
        \left(
            \displaystyle\int_{\mathbb D}
            \big(x(s)-x(1)\big)^2\,\mathrm{d}s
        \right)^{1/2}
    }\,,
\]
which is continuous whenever its denominator is strictly positive, yields

$$
    \frac{\Phi_N(1)-\Phi}
    {\left(\displaystyle\int_{\mathbb D}
    \big(\Phi_N(s)-\Phi_N(1)\big)^2\,\mathrm{d}s\right)^{\!1/2}}
    ~\stackrel{d}{\longrightarrow}~
    \frac{\tau_{w_f}\mathbb B(1)}
    {\tau_{w_f}\left(\displaystyle\int_{\mathbb D}
    \frac{1}{s^2}\big(\mathbb B(s)-s\mathbb B(1)\big)^2
    \,\mathrm{d}s\right)^{\!1/2}}\,.
$$
Notice that the denominator on the right-hand side is almost surely positive
\citep{QuanzDette2025}. Since $\tau_{w_f}>0$, it cancels, which establishes the claimed pivotal weak limit. \hfill\qed

\subsection{Proof of Proposition \ref{Prop1}}\label{Sec42}
We establish
\begin{align}
    \boldsymbol{\mathcal G}_N
    =
    \big\{\mathcal G_N(s)\big\}_{s\in\mathbb D}
    \coloneqq
    \big\{\sqrt{N}\Psi_N^{\circ}(s)\big\}_{s\in\mathbb D}
    \,\stackrel{d}{\longrightarrow}\,
    \boldsymbol{\mathcal G}
    =
    \big\{\mathcal G(s)\big\}_{s\in\mathbb D},
\end{align}
in $\ell^\infty(\mathbb D)$, where $\Psi_N^{\circ}$ is defined in
\eqref{eq14} and $\mathcal G(s)$ in \eqref{eq9}.
As in the proof of Theorem~\ref{Thm1} (Section~\ref{Sec41}), we work with
the more tractable process $\widetilde{\boldsymbol{\mathcal G}}_N$, obtained
from $\boldsymbol{\mathcal G}_N$ by replacing $\Psi_N$ with
$\widetilde{\Psi}_N$ in \eqref{eq14}. Since the remaining arguments proceed
analogously to those in the proof of Theorem~\ref{Thm1}, except that no
self-normalization is required, we focus on the parts that require a separate
derivation: the approximation (Step~1)
\begin{align}\label{eq29}
    \sup_{s\in\mathbb D}
    \big|\widetilde{\Psi}_N(s)-\Psi_N(s)\big|
    &=
    o_{\Prob}\big(N^{-1/2}\big),
\end{align}
and the asymptotic first- and second-order moment structure (Steps~2--3).

\paragraph{Step~1.} In the following, we show \eqref{eq29}. As in step 1 in Section \ref{Sec41}, let
$m\in\{\lfloor N/2\rfloor,\ldots,N\}$ and recall from
\eqref{eq19} that
\[
    I_m(\lambda)
    =
    \widetilde I_m(\lambda)+R_m(\lambda).
\]
Consequently,
\begin{align}\label{eq30}
    I_m(\lambda_j)I_m(\lambda_{j-1})
    -
    \widetilde I_m(\lambda_j)
    \widetilde I_m(\lambda_{j-1})
    =
    R_m(\lambda_j)\widetilde I_m(\lambda_{j-1})
    +
    \widetilde I_m(\lambda_j)R_m(\lambda_{j-1})
    +
    R_m(\lambda_j)R_m(\lambda_{j-1}).
\end{align}
We control the three terms on the right-hand side separately.

\bigskip

\noindent{\bf Step~1.1.} We first consider
\[
    \sum_{j=1}^{\lfloor N/2\rfloor}
    R_m(\lambda_j)\widetilde I_m(\lambda_{j-1}).
\]
Since
\[
    \widetilde I_m(\lambda)
    =
    |\Xi(\lambda)|^2
    |J_{m,\varepsilon}(\lambda)|^2,
\]
with $\Xi(\lambda)$ in \eqref{eq16}, and, by \eqref{eq19},
\[
    R_m(\lambda)
    =
    2\Re\!\left(
        \Xi(\lambda)J_{m,\varepsilon}(\lambda)
        \overline{Y_m(\lambda)}
    \right)
    +
    |Y_m(\lambda)|^2,
\]
it suffices first to control
\begin{align}\label{eq31}
    \sum_{j=1}^{\lfloor N/2\rfloor}
    \Xi(\lambda_j)J_{m,\varepsilon}(\lambda_j)
    \overline{Y_m(\lambda_j)}
    |\Xi(\lambda_{j-1})|^2
    |J_{m,\varepsilon}(\lambda_{j-1})|^2.
\end{align}

For this purpose, define
\begin{align}
    \eta_{N,k}
    \coloneqq
    \frac{1}{N}
    \sum_{j=1}^{\lfloor N/2\rfloor}
    \Xi(\lambda_j)
    |\Xi(\lambda_{j-1})|^2
    \e^{-\imag k\lambda_j},
    \qquad k\in\mathbb Z.
\end{align}
By construction, $\eta_{N,k}$ is $N$-periodic in $k$. Moreover, Parseval's identity and boundedness of $\Xi$ give
\begin{align}\label{eq32}
    \sum_{k=0}^{N-1}
    |\eta_{N,k}|^2
    &=
    \frac{1}{N}
    \sum_{j=1}^{\lfloor N/2\rfloor}
    |\Xi(\lambda_j)|^2
    |\Xi(\lambda_{j-1})|^4
    = O(1),
\end{align}
uniformly in $N$.

Recall from Step~1.1 that
\[
    U_{m,\ell}(\lambda)
    =
    \sum_{r\in B_{m,\ell}}
    c_{m,\ell,r}\varepsilon_r
    \e^{-\imag r\lambda},
    \qquad
    |B_{m,\ell}|
    \leq
    2\min(|\ell|,m),
    \qquad
    c_{m,\ell,r} \in \{-1,1\}.
\]
Expanding the terms in \eqref{eq31} and using
$\lambda_{j-1}=\lambda_j-2\pi/N$, yields
\begin{align}
    &\sum_{j=1}^{\lfloor N/2\rfloor}
    \Xi(\lambda_j)J_{m,\varepsilon}(\lambda_j)
    \overline{Y_m(\lambda_j)}
    |\Xi(\lambda_{j-1})|^2
    |J_{m,\varepsilon}(\lambda_{j-1})|^2
    \nonumber\\
    &\qquad=
    \frac{N}{(2\pi m)^2}
    \sum_{\ell=-\infty}^{\infty}
    \psi_\ell
    \sum_{r\in B_{m,\ell}}
    c_{m,\ell,r}\,
    \varepsilon_r
    \sum_{t,u,v=1}^{m}
    \varepsilon_t\varepsilon_u\varepsilon_v
    \e^{2\pi\imag(u-v)/N}
    \eta_{N,t-\ell-r+u-v}.
\end{align}
For fixed $m,\ell,$ and $r$, put
\begin{align*}
    Z_{m,\ell,r}
    \coloneqq
    \sum_{t,u,v=1}^{m}
    \varepsilon_t\varepsilon_u\varepsilon_v
    \e^{2\pi\imag(u-v)/N}
    \eta_{N,t-\ell-r+u-v}.
\end{align*}
Since $Z_{m,\ell,r}$ is a polynomial of degree at most three in the Gaussian innovations, Gaussian hypercontractivity yields, for every fixed $q\geq1$,
\begin{align*}
    \nu_{2q}(Z_{m,\ell,r})
    \,\lesssim\,
    \nu_2(Z_{m,\ell,r}),
\end{align*}
while, by Isserlis's theorem, \eqref{eq32}, and the
$N$-periodicity of $\eta_{N,r}$ in $r,$
\begin{align*}
    \nu_2(Z_{m,\ell,r})^2
    \,\lesssim\,
    m^2
    \sum_{k=0}^{N-1}
    |\eta_{N,k}|^2
    \,\lesssim\,
    m^2,
\end{align*}
uniformly in
$\lfloor N/2\rfloor\leq m\leq N$ and $\ell,r\in\mathbb Z,$ and therefore
\[
    \nu_{2q}(Z_{m,\ell,r})
    \,\lesssim\,
    m
\]
uniformly in $m$ and $\ell, r.$ Hence, by Minkowski's inequality, followed by H\"older's inequality, we obtain
\begin{align*}
    \sup_{\lfloor N/2\rfloor\leq m\leq N}\nu_q\Bigg(\sum_{j=1}^{\lfloor N/2\rfloor}
        \Xi(\lambda_j)J_{m,\varepsilon}(\lambda_j)
        \overline{Y_m(\lambda_j)}
        \widetilde I_m(\lambda_{j-1})
    \Bigg)
    &\lesssim
    \frac{N}{m^2}
    \sum_{\ell=-\infty}^{\infty}
    |\psi_\ell|
    \sum_{r\in B_{m,\ell}}
    \nu_{2q}(\varepsilon_r)
    \nu_{2q}(Z_{m,\ell,r})
    \\
    &\lesssim
    \frac{N}{m}
    \sum_{\ell=-\infty}^{\infty}
    \min(|\ell|,m)|\psi_\ell|
    \\
    &\lesssim
    N^{1-\alpha},
\end{align*}
where we used $m\asymp N,$ $\min(|\ell|,N) \leq N^{1-\alpha}|\ell|^\alpha,$ and the summability condition in \eqref{eq2}.

It remains to consider the quadratic part of $R_m$. By Step~1.1,
for every fixed $q\geq1$,
\[
    \sup_{\lambda\in[-\pi,\pi]}
    \sup_{\lfloor N/2\rfloor\leq m\leq N}
    \nu_{4q}\big(Y_m(\lambda)\big)
    =
    O(N^{-1/2}).
\]
Moreover, since
\[
    \widetilde I_m(\lambda)
    =
    |\Xi(\lambda)|^2
    |J_{m,\varepsilon}(\lambda)|^2,
\]
boundedness of $\Xi$ and Gaussianity of $J_{m,\varepsilon}$ imply
\[
    \sup_{\lambda\in[-\pi,\pi]}
    \sup_{\lfloor N/2\rfloor\leq m\leq N}
    \nu_{2q}\big(\widetilde I_m(\lambda)\big)
    = O(1),
\]
uniformly in $N.$ Thus, by Minkowski's and H\"older's inequalities,
\begin{align*}
    \sup_{\lfloor N/2\rfloor\leq m\leq N}
    \nu_q\Bigg(
        \sum_{j=1}^{\lfloor N/2\rfloor}
        |Y_m(\lambda_j)|^2
        \widetilde I_m(\lambda_{j-1})
    \Bigg)
    &\lesssim
    \sum_{j=1}^{\lfloor N/2\rfloor}
    \sup_{\lfloor N/2\rfloor\leq m\leq N}
    \nu_{2q}\big(|Y_m(\lambda_j)|^2\big)
    \nu_{2q}\big(\widetilde I_m(\lambda_{j-1})\big)\\
    &\lesssim
    \sum_{j=1}^{\lfloor N/2\rfloor}
    \sup_{\lfloor N/2\rfloor\leq m\leq N}
    \nu_{4q}\big(Y_m(\lambda_j)\big)^2
    =
    O(1).
\end{align*}
Combining the preceding bounds and using $|\Re(z)|\leq|z|$ gives
\begin{align}
    \sup_{\lfloor N/2\rfloor\leq m\leq N}
    \nu_q\Bigg(
        \sum_{j=1}^{\lfloor N/2\rfloor}
        R_m(\lambda_j)
        \widetilde I_m(\lambda_{j-1})
    \Bigg)
    =
    O(N^{1-\alpha}).
\end{align}

\paragraph{Step~1.2.} The second cross term in \eqref{eq30} is treated in the
same way up to a phase shift. Hence,
\begin{align}
    \sup_{\lfloor N/2\rfloor\leq m\leq N}
    \nu_q\Bigg(
        \sum_{j=1}^{\lfloor N/2\rfloor}
        \widetilde I_m(\lambda_j)
        R_m(\lambda_{j-1})
    \Bigg)
    =
    O(N^{1-\alpha}).
\end{align}

\paragraph{Step~1.3.} Finally, we consider the product of the two remainder terms in \eqref{eq30}. From
\eqref{eq19}, H\"older's inequality, the uniform moment bounds
for $J_{m,\varepsilon}$, and Step~1.1,
\begin{align*}
    \sup_{\lambda\in[-\pi,\pi]}
    \sup_{\lfloor N/2\rfloor\leq m\leq N}
    \nu_{2q}\big(R_m(\lambda)\big)
    =
    O(N^{-1/2}).
\end{align*}
Therefore, by Minkowski's and H\"older's inequalities,
\begin{align}\label{eq33}
    \sup_{\lfloor N/2\rfloor\leq m\leq N}
    \nu_q\Bigg(
        \sum_{j=1}^{\lfloor N/2\rfloor}
        R_m(\lambda_j)R_m(\lambda_{j-1})
    \Bigg)
    &=
    O(1).
\end{align}
Choose $q$ sufficiently large such that $1/q < \alpha - 1/2.$ Since there are at most $N$ possible values of $m$, it follows from
\eqref{eq30}--\eqref{eq33} that
\begin{align*}
    &\nu_q\Bigg(
        \max_{\lfloor N/2\rfloor\leq m\leq N}
        \Bigg|
        \sum_{j=1}^{\lfloor N/2\rfloor}
        \Big[
            I_m(\lambda_j)I_m(\lambda_{j-1})
            -
            \widetilde I_m(\lambda_j)
            \widetilde I_m(\lambda_{j-1})
        \Big]
        \Bigg|
    \Bigg)
    \\
    &\qquad\leq
    N^{1/q}
    \sup_{\lfloor N/2\rfloor\leq m\leq N}
    \nu_q\Bigg(
        \sum_{j=1}^{\lfloor N/2\rfloor}
        \Big[
            I_m(\lambda_j)I_m(\lambda_{j-1})
            -
            \widetilde I_m(\lambda_j)
            \widetilde I_m(\lambda_{j-1})
        \Big]
    \Bigg)
    \\
    &\qquad=
    O\big(N^{1-\alpha+1/q}\big)
    +
    O(N^{1/q})
    =
    o(N^{1/2}).
\end{align*}
Using Lyapunov's inequality and Markov's inequality, we conclude that
\begin{align}
    \max_{\lfloor N/2\rfloor\leq m\leq N}
    \Bigg|
        \sum_{j=1}^{\lfloor N/2\rfloor}
        \Big[
            I_m(\lambda_j)I_m(\lambda_{j-1})
            -
            \widetilde I_m(\lambda_j)
            \widetilde I_m(\lambda_{j-1})
        \Big]
    \Bigg|
    =
    o_{\Prob}(N^{1/2}).
\end{align}

Finally, let $N_s=\lfloor sN\rfloor$. By the definitions of
$\Psi_N$ and $\widetilde\Psi_N$, and since $N_s/(sN)\leq1$ for
$s\in\mathbb D=[1/2,1]$, the same argument as at the end of
Step~1.3 gives
\begin{align*}
    \sup_{s\in\mathbb D}
    \big|
        \Psi_N(s)-\widetilde\Psi_N(s)
    \big|
    &\leq
    \frac{4\pi}{N}
    \max_{\lfloor N/2\rfloor\leq m\leq N}
    \Bigg|
        \sum_{j=1}^{\lfloor N/2\rfloor}
        \Big[
            I_m(\lambda_j)I_m(\lambda_{j-1})
            -
            \widetilde I_m(\lambda_j)
            \widetilde I_m(\lambda_{j-1})
        \Big]
    \Bigg|
    \\
    &=
    o_{\Prob}(N^{-1/2}),
\end{align*}
which proves \eqref{eq29}.

\bigskip 

\paragraph{Step~2.} Here, we show that
\begin{align}\label{eq34}
    \sup_{s\in\mathbb D}\big|\Exp(\widetilde\Psi^{\circ}_N(s))\big| = o(N^{-1/2}),
\end{align}
and recall that $\Psi(s) = (1 + \frac{\sin^2(\pi s)}{\pi^2s^2})\Psi.$ The proof utilizes several arguments from Steps 2--3 in Section \ref{Sec41}. As $(\varepsilon_k)$ is a Gaussian white noise with unit variance, Isserlis's theorem yields
$\Exp(\varepsilon_m\varepsilon_n\varepsilon_o\varepsilon_p) = \delta_{mn}\delta_{op} + \delta_{mo}\delta_{np} + \delta_{mp}\delta_{no},$ leading to
\begin{align*}
    \Exp\!\big(I^s_{N,\varepsilon}(\lambda_j)I^s_{N,\varepsilon}(\lambda_{j-1})\big)
    &= \frac{1}{4\pi^2 s^2N^2}\sum_{m,n,o,p=1}^{N_s}\Exp(\varepsilon_m\varepsilon_n\varepsilon_o\varepsilon_p)\e^{-\imag[(m-n)\lambda_j + (o-p)\lambda_{j-1}]}\\
    &= \frac{1}{4\pi^2 s^2N^2}\bigg(N^2_s + \sum_{m,n=1}^{N_s}\Big[\e^{-\imag(m-n)\lambda_{2j-1}} + \e^{-\imag(m-n)\lambda_1}\Big] - 2N_s\bigg)\allowdisplaybreaks\\
    &= \frac{N_s}{4\pi^2 s^2N^2}\Big(N_s + 
    F_{N_s}(\lambda_{2j-1}) + F_{N_s}(\lambda_1)  - 2
    \Big)\\
    &=
    \frac{1}{4\pi^2}\Bigg(\,1 + 
    \frac{N_s}{s^2N^2}
    \Big(
    F_{N_s}(\lambda_{2j-1})
    +F_{N_s}(\lambda_1)
    \Big)
    + \underbrace{\Big(\frac{N_s^2}{s^2N^2}
    -1\Big)
    -\frac{2N_s}{s^2N^2}\,}_{=\,O(N^{-1})\,}\;
    \Bigg),
\end{align*}
where the $O(N^{-1})$ term is uniform in $j$ and $s.$ Subsequently, since $f$ is symmetric,
\begin{align*}
    \Exp\!\big(\widetilde\Psi^{\circ}_N(s)\big)
    &= \frac{16\pi^3}{N}\sum_{j=1}^{\lfloor N/2\rfloor}
    f(\lambda_j)f(\lambda_{j-1})
    \Exp\!\big(I^s_{N,\varepsilon}(\lambda_j)I^s_{N,\varepsilon}(\lambda_{j-1})\big)
    - \bigg(1 + \frac{\sin^2(\pi s)}{\pi^2s^2}\bigg)\int^\pi_{-\pi}f^2(\lambda)\,\mathrm{d}\lambda
    \notag\\
    &= \frac{16\pi^3}{N}\sum_{j=1}^{\lfloor N/2\rfloor}
    f^2(\lambda_j)
    \Exp\!\big(I^s_{N,\varepsilon}(\lambda_j)I^s_{N,\varepsilon}(\lambda_{j-1})\big)
    - \bigg(1 + \frac{\sin^2(\pi s)}{\pi^2s^2}\bigg)\int^\pi_{-\pi}f^2(\lambda)\,\mathrm{d}\lambda
    +O(N^{-\alpha})
    \notag\\
    &= 2\Bigg(
    \frac{2\pi}{N}\sum_{j=1}^{\lfloor N/2\rfloor}
    f^2(\lambda_j)
    - \int_0^\pi f^2(\lambda)\,\mathrm{d}\lambda
    \Bigg)\notag\\
    &\qquad\quad\! + 2\Bigg[\frac{2\pi N_s}{s^2N^3}\sum_{j=1}^{\lfloor N/2\rfloor}f^2(\lambda_j)
    \Big(
    F_{N_s}(\lambda_{2j-1})
    +F_{N_s}(\lambda_1)
    \Big) - \frac{\sin^2(\pi s)}{\pi^2s^2}\int^\pi_0f^2(\lambda)\,\mathrm{d}\lambda\Bigg]
    +O(N^{-1})\,,
\end{align*}
uniformly in $s.$ For the first term, since $f$ is H\"older continuous of order $\alpha,$ which transfers to $f^2,$ and since $\lambda_j - \lambda_{j-1} = 2\pi/N,$ we obtain (cf.\ Step 2 in Section \ref{Sec41})
\begin{align*}
    \frac{2\pi}{N}\sum_{j=1}^{\lfloor N/2\rfloor}
    f^2(\lambda_j)
    - \int_0^\pi f^2(\lambda)\,\mathrm{d}\lambda
    &=
    \sum_{j=1}^{\lfloor N/2\rfloor}\int^{\lambda_j}_{\lambda_{j-1}}
    \big[f^2(\lambda_j)
    -  f^2(\lambda)\big]\,\mathrm{d}\lambda\notag\\
    &= O(N^{-\alpha})\,.
\end{align*}
Next, for the second summand, boundedness of $f$, Lemma \ref{Lem2}~(a), and $N_s\asymp N$ yield
\begin{align*}
\sum_{j=1}^{\lfloor N/2\rfloor}
f^2(\lambda_j)F_{N_s}(\lambda_{2j-1})
&\lesssim
\frac{N^2}{N_s}
\sum_{j=1}^{\lfloor N/2\rfloor}
\frac{1}{\min(2j-1,N-(2j-1))^2}
=
O(N),
\end{align*}
uniformly in $s$. Since this term is multiplied by $N_s/N^3$, its contribution is $O(N^{-1})$. Finally, since
\[
    N_sF_{N_s}(\lambda_1)
    =\frac{\sin^2(N_s\lambda_1/2)}{\sin^2(\lambda_1/2)}
    \sim \frac{N^2}{\pi^2}\sin^2(\pi s),
    \qquad s\in\mathbb D=[1/2,1],
\]
the symmetry and H\"older continuity of $f$ of order $\alpha,$ and hence of $f^2,$ yield, analogously to the above,
\begin{align*}
    &\frac{4\pi N_s}{s^2N^3}
    \sum_{j=1}^{\lfloor N/2\rfloor}f^2(\lambda_j)F_{N_s}(\lambda_1)
    -\frac{\sin^2(\pi s)}{\pi^2s^2}
    \int_{-\pi}^{\pi}f^2(\lambda)\,\mathrm{d}\lambda\\
    & =
    2\Bigg(
    \frac{\sin^2(N_s\lambda_1/2)}{\sin^2(\lambda_1/2)}
    \frac{4\pi}{s^2N^3}
    \sum_{j=1}^{\lfloor N/2\rfloor}f^2(\lambda_j)
    -\frac{\sin^2(\pi s)}{\pi^2s^2}
    \int_0^\pi f^2(\lambda)\,\mathrm{d}\lambda
    \Bigg)\\
    &=O(N^{-\alpha})\,,
\end{align*}
uniformly in $s.$ Altogether, since $\alpha \in (1/2,1],$ \eqref{eq34} follows.

\paragraph{Step~3.} First, for any $h,j,s,t$,
Isserlis's theorem yields
\begin{align}
    &\Cov\!\big(
    I^s_{N,\varepsilon}(\lambda_h)I^s_{N,\varepsilon}(\lambda_{h-1}),
    I^t_{N,\varepsilon}(\lambda_j)I^t_{N,\varepsilon}(\lambda_{j-1})
    \big)\notag\\
    &= \frac{1}{s^2t^2N^4}
    \sum_{m,n,o,p=1}^{N_s}\sum_{q,r,u,v=1}^{N_t}
    \sum_{\pi}\prod_{(a,b)\in\pi}\delta_{ab}\,
    \e^{-\imag[(m-n)\lambda_h+(o-p)\lambda_{h\!-\!1}
    +(q-r)\lambda_j+(u-v)\lambda_{j-1}]},
    \label{eq35}
\end{align}
where $\pi$ runs over pair partitions of $\{m,n,o,p,q,r,u,v\}$
with at least one cross-pair between $\{m,n,o,p\}$ and
$\{q,r,u,v\}$. There are
$(8-1)!!-(4-1)!!^2=96$ such partitions, which we classify into four types. Note that the (non-sequential) periodogram equals the product of the DFT and its conjugate. For Gaussian white noise innovations, the Leonov--Shiryaev formula for cumulants of products \citep{LeonovShiryaev1959}; see also \citet[][Section 2.3]{Brillinger2001}, then yields the result relatively directly. In the sequential setting considered here, however, the sums do not extend over all Fourier frequencies $\lambda_j=2\pi j/N$, $j\in\{0,\ldots,N-1\}$, except when $s=1$, making the calculations more tedious.

First, $8$ partitions satisfy $m=n$ or $o=p$ and, simultaneously,
$q=r$ or $u=v$. They yield Fejér kernels with factors of order $N^3$.
Second, $24$ partitions contain only cross-pairs and yield products of
four upper Dirichlet kernels whose arguments depend on both $h$ and $j$.
For example, $m=q$, $n=r$, $o=u$, and $p=v$ gives a four-fold product of upper Dirichlet kernels $D^+_K$ (see the beginning of this section). Third, $32$ partitions have an internal pairing in exactly one of the
two sets. They yield products of three upper Dirichlet kernels with
factors of order $N$, with at least one argument depending only on $h$
or $j$. Their total contribution is $o(N^5)$. Finally, the
remaining $32$ partitions contain one internal pairing in each set and
yield products of four upper Dirichlet kernels, which are likewise
asymptotically negligible. For instance, for the internal pairs $m=o$
and $q=u$, together with $n=r$ and $p=v$, H\"older's inequality, Lemma \ref{Lem1}, and the arguments in Section \ref{Sec41} give
\begin{align*}
    &\sum_{h,j=1}^{\lfloor N/2\rfloor}
    \big|D^+_{N_{st}}(\lambda_{2h-1})D^+_{N_{st}}(\lambda_{2j-1})D^+_{N_{st}}(\lambda_{-h-j})D^+_{N_{st}}(\lambda_{2-h-j})
    \big|\\
    &\leq
    \bigg(
    \sum_{h,j=1}^{\lfloor N/2\rfloor}
    \big|D^+_{N_{st}}(\lambda_{2h-1})D^+_{N_{st}}(\lambda_{2j-1})\big|^2
    \bigg)^{1/2}
    \bigg(
    \sum_{h,j=1}^{\lfloor N/2\rfloor}
    \big|D^+_{N_{st}}(\lambda_{-h-j})\big|^4
    \bigg)^{1/4}
    \bigg(
    \sum_{h,j=1}^{\lfloor N/2\rfloor}
    \big|D^+_{N_{st}}(\lambda_{2-h-j})\big|^4
    \bigg)^{1/4}\\
    &= o(N^5)\,.
\end{align*}

Hence, the quantity in \eqref{eq35} asymptotically coincides with $32$ terms (grouped where possible), corresponding to the contributions from the $8$ Fejér-type partitions in the first case and the $24$ fully cross-linked partitions described in the second case and contained in $D^{sstt}_N(h,j)$. More precisely,
\begin{align}
    \Cov\!\Big(
    I^s_{N,\varepsilon}(\lambda_h)I^s_{N,\varepsilon}(\lambda_{h-1}),
    I^t_{N,\varepsilon}(\lambda_j)I^t_{N,\varepsilon}(\lambda_{j-1})
    \Big)
    &\sim \frac{1}{s^2t^2N^4}
    \Big[N_sN_tN_{st}\,S_{N_{st}}(h,j)+D^{sstt}_N(h,j)\Big],
    \label{eq36}
\end{align}
where $S_{N_{st}}(h,j)$ denotes the following linear combination of single Fejér kernels:
\begin{align}\label{eq37}
\begin{split}
    S_{N_{st}}(h,j)
    &\coloneqq
    F_{N_{st}}(\lambda_{h-j-1})
    +2F_{N_{st}}(\lambda_{h-j})
    +F_{N_{st}}(\lambda_{h-j+1})\\[0.5ex]
    &\quad\;\,
    +F_{N_{st}}(\lambda_{h+j-2})
    +2F_{N_{st}}(\lambda_{h+j-1})
    +F_{N_{st}}(\lambda_{h+j})\,.
\end{split}
\end{align}
Moreover, $D^{sstt}_N(h,j)$, which, due to
$D^+_{N_{st}}(x)D^+_{N_{st}}(-x)=N_{st}F_{N_{st}}(x)$, can be decomposed as
\begin{align}
D^{sstt}_N(h,j)
&=N_{st}^2\Big(
T_{N_{st},1}(h,j)
+T_{N_{st},2}(h,j)
+T_{N_{st},3}(h,j)
+T_{N_{st},4}(h,j)
\Big)
+R^{sstt}_N(h,j),
\label{eq38}
\end{align}
where the individual terms are obtained by grouping the corresponding products in the explicit representation above. In particular,
\begin{align}
    T_{N_{st},1}(h,j)
    &\coloneqq F^2_{N_{st}}(\lambda_{h-j}),
    \label{eq39}\\
    T_{N_{st},2}(h,j)
    &\coloneqq
    F_{N_{st}}(\lambda_{h-j-1})
    \big[
        F_{N_{st}}(\lambda_{h-j})
        +
        F_{N_{st}}(\lambda_{h-j+1})
    \big],
    \label{eq40}\\
    T_{N_{st},3}(h,j)
    &\coloneqq
    F_{N_{st}}(\lambda_{h-j})
    \big[
        F_{N_{st}}(\lambda_{h+j-2})
        +
        F_{N_{st}}(\lambda_{h+j-1})
        +
        2F_{N_{st}}(\lambda_{h+j})
    \big],
    \label{eq41}\\
    T_{N_{st},4}(h,j)
    &\coloneqq
    F_{N_{st}}(\lambda_{h+j})
    \big[
        F_{N_{st}}(\lambda_{h-j+1})
        +
        F_{N_{st}}(\lambda_{h+j-2})
        +
        F_{N_{st}}(\lambda_{h+j-1})
    \big].
    \label{eq42}
\end{align}
The term $R^{sstt}_N(h,j)$ contains the remaining products of four upper Dirichlet kernels.

For the limiting covariances of the quantities $\widetilde \Psi_N$ in Section \ref{Sec42}, with $N_s=\lfloor sN\rfloor$ and $N_{st}=\lfloor\min(s,t)N\rfloor$, the arguments above, together with \eqref{eq36}, yield with $f_{[1]}(\lambda_h)=f(\lambda_h)f(\lambda_{h-1}),$
\begin{align}
    N\Cov\!\big(\widetilde \Psi^{\circ}_N(s),\widetilde \Psi^{\circ}_N(t)\big)
    &=\frac{16\pi^2}{N}
    \sum_{h,j=1}^{\lfloor N/2\rfloor}
    f_{[1]}(\lambda_h)f_{[1]}(\lambda_j)\Cov\!\big(
    I^s_{N,\varepsilon}(\lambda_h)I^s_{N,\varepsilon}(\lambda_{h-1}),
    I^t_{N,\varepsilon}(\lambda_j)I^t_{N,\varepsilon}(\lambda_{j-1})
    \big)\notag\allowdisplaybreaks\\
    &\sim \frac{4\pi^2}{\max(s,t)N^2}
    \sum_{h,j=0}^{N-1}
    f_{[1]}(\lambda_h)f_{[1]}(\lambda_j)
    S_{N_{st}}(h,j)
    \label{eq43}\\
    &\quad\;\,
    +\frac{16\pi^2}{s^2t^2N^5}
    \sum_{h,j=1}^{\lfloor N/2\rfloor}
    f_{[1]}(\lambda_h)f_{[1]}(\lambda_j)
    D^{sstt}_N(h,j),
    \label{eq44}
\end{align}
\noindent with $D^{sstt}_N(h,j)$ from \eqref{eq38}. By the definition of $S_{N_{st}}(h,j)$ in \eqref{eq37} and the arguments of Section \ref{Sec41}, the term in \eqref{eq43} satisfies
\begin{align}
    \frac{4\pi^2}{\max(s,t)N^2}
    \sum_{h,j=0}^{N-1}
    f_{[1]}(\lambda_h)f_{[1]}(\lambda_j)
    S_{N_{st}}(h,j)
    \,\stackrel{N\to\infty}{\longrightarrow}\,
    \frac{16\pi}{\max(s,t)}
    \int^\pi_{-\pi}f^4(\lambda)\,\mathrm{d}\lambda.
    \label{eq45}
\end{align}

Next, we determine the limit of the term in \eqref{eq44}. To this end, we first establish the following preparatory approximation result.

\bigskip

\noindent{\bf Step 3.1.} For $m=1,2$, define
$$
    \widetilde{\sigma}_{N_{st},m}(g,\lambda_h)
    \coloneqq
    \frac{1}{N}\sum_{j=0}^{N-1}
    g(\lambda_j)\,
    \widetilde T_{N_{st},m}(h,j),
    \qquad
    \widetilde T_{N_{st},m}(h,j)
    \coloneqq
    \frac{T_{N_{st},m}(h,j)}
    {\displaystyle\frac{1}{N}\sum_{k=0}^{N-1}T_{N_{st},m}(k,0)}.
$$
We show that for $f_{[1]},$ with $f_{[1]}(\lambda_j) = f(\lambda_j)f(\lambda_{j-1}),$ it holds that 
\begin{equation}
    \max_{0\leq h<N}
    \big|
    \widetilde{\sigma}_{N_{st},m}(f_{[1]},\lambda_h)
    -f_{[1]}(\lambda_h)
    \big|
    =o(1),
\qquad m=1,2.
\label{eq46}
\end{equation}

By the definition of $T_{N_{st},m},$ it holds that
$$
    \sum_{j=0}^{N-1}
    T_{N_{st},m}(h,j)
    \,=\,
    \sum_{k=0}^{N-1}
    T_{N_{st},m}(k,0),  \qquad m=1,2.
$$
Consequently,
$$
    \frac{1}{N}\sum_{j=0}^{N-1}
    \widetilde T_{N_{st},m}(h,j)=1,  \qquad m=1,2,
$$
and therefore
\begin{equation}
    \widetilde{\sigma}_{N_{st},m}(g,\lambda_h)
    -g(\lambda_h)
    =
    \frac{1}{N}\sum_{j=0}^{N-1}
    \bigl(g(\lambda_j)-g(\lambda_h)\bigr)
    \widetilde T_{N_{st},m}(h,j), \qquad m=1,2.
\end{equation}

We first consider $m=1$. By the normalization of
$\widetilde T_{N_{st},1}$ and using ideas in the
proof of Lemma~\ref{Lem3},
\begin{align*}
    \big|
        \widetilde{\sigma}_{N_{st},1}(f_{[1]},\lambda_h)
        -f_{[1]}(\lambda_h)
    \big|
    &\lesssim
    \frac{
        \displaystyle
        \frac{1}{N}
        \sum_{0<|k|\leq \lfloor N/2\rfloor}
        |\lambda_k|^\alpha
        F_{N_{st}}^2(\lambda_k)
    }{
        \displaystyle
        \frac{1}{N}
        \sum_{k=0}^{N-1}
        F_{N_{st}}^2(\lambda_k)
    }.
\end{align*}
By Lemma~\ref{Lem2}~(a),
\[
    F_{N_{st}}(\lambda_k)
    \lesssim
    \frac{N^2}{N_{st}k^2},
    \qquad 0<|k|\leq \lfloor N/2\rfloor,
\]
whereas, using $F_{N_{st}}(0)=N_{st}$,
\[
    \frac{1}{N}\sum_{k=0}^{N-1}
    F_{N_{st}}^2(\lambda_k)
    \geq
    \frac{N_{st}^2}{N}.
\]
Therefore, since $f$ is H\"older continuous of order $\alpha \in (1/2,1],$ the function $f_{[1]}$ is too, leading to 
\begin{align*}
    \big|
        \widetilde{\sigma}_{N_{st},1}(f_{[1]},\lambda_h)
        -f_{[1]}(\lambda_h)
    \big|
    \,\lesssim\,
    \frac{N^{4-\alpha}}{N_{st}^4}
    \sum_{k=1}^{N} k^{\alpha-4}
    \,\lesssim\,
    \frac{N^{4-\alpha}}{N_{st}^4}
    \,=\, o(1),
\end{align*}
uniformly in $h$, since $N_{st}\asymp N$. This proves
\eqref{eq46} for $m=1$.

Next, consider $m=2.$ By the definition of $T_{N_{st},2}(h,j)$ in \eqref{eq40} and $2|ab|\leq a^2+b^2$, we have
$$
    |T_{N_{st},2}(h,j)|
    \leq
    F_{N_{st}}^2(\lambda_{h-j-1})
    +\frac{1}{2}F_{N_{st}}^2(\lambda_{h-j})
    +\frac{1}{2}F_{N_{st}}^2(\lambda_{h-j+1}).
$$
Thus, $T_{N_{st},2}$ is dominated by a finite sum of squared Fejér kernels whose arguments differ from $\lambda_{h-j}$ by at most $2\pi/N.$ Hence, the same arguments as for $m=1$ yield \eqref{eq46} for $m=2$.

\bigskip

\noindent{\bf Step 3.2.} Here, we consider the quantity $T_{N_{st},1}(h,j)=F^2_{N_{st}}(\lambda_{h-j})$ in \eqref{eq39}. As $\sum_{k=0}^{N-1}\e^{-\imag k\lambda}=N$ if $\lambda\equiv0 \bmod{2\pi}$ and zero otherwise, elementary manipulations and standard formulas for finite sums yield
\begin{align}
    \frac{1}{N}\sum_{k=0}^{N-1}T_{N_{st},1}(k,0)
    &=\frac{1}{N}\!\sum_{|\ell|,|m|<N_{st}}
    \!\Big(1-\frac{|\ell|}{N_{st}}\Big)
    \Big(1-\frac{|m|}{N_{st}}\Big)
    \sum_{k=0}^{N-1}\e^{-\imag(\ell+m)\lambda_k}
    \notag\\[0.5ex]
    &=\sum_{\substack{|\ell|,|m|<N_{st}\\
    \ell+m\in\{-N,0,N\}}}
    \!\Big(1-\frac{|\ell|}{N_{st}}\Big)
    \Big(1-\frac{|m|}{N_{st}}\Big)
    \notag\allowdisplaybreaks\\[0.5ex]
    &=\sum_{\ell=1-N_{st}}^{N_{st}-N-1}
    \Big(1-\frac{|\ell|}{N_{st}}\Big)
    \Big(1-\frac{|N+\ell|}{N_{st}}\Big)
    +\sum_{|\ell|<N_{st}}
    \!\Big(1-\frac{|\ell|}{N_{st}}\Big)^2
    \notag\allowdisplaybreaks\\[1ex]
    &=\frac{2N_{st}^2+1}{3N_{st}}
    +\frac{(2N_{st}-N-1)(2N_{st}-N)(2N_{st}-N+1)}
    {3N_{st}^2}
    \notag\\[1ex]
    &\sim
    \frac{2\min(s,t)N}{3}
    +\frac{(2\min(s,t)-1)^3N}{3\min^2(s,t)}\,.
\end{align}
Here, for the latter sum, we used that $2N_{st}>N+1$ for all sufficiently large $N$ whenever $\min(s,t)>1/2$. The boundary case $\min(s,t)=1/2$ does not alter the resulting formula, as the corresponding additional term vanishes. In fact, this calculation motivates the restriction to $\mathbb{D}=[1/2,1]$\label{eq47}. Further, for even $N$ (the case of odd $N$ can be shown similarly), by symmetry, $2\pi$-periodicity, and non-negativity of the Fejér kernel, Lemma \ref{Lem2}(a) gives, for $T_{N_{st},1}(h,j)=F^2_{N_{st}}(\lambda_{h-j})$,
\begin{align*}
    \sum_{h=\lfloor N/2\rfloor+1}^{N}
    \sum_{j=1}^{\lfloor N/2\rfloor}
    T_{N_{st},1}(h,j)
    &=
    2\sum_{k=1}^{N/2-1}kF^2_{N_{st}}(\lambda_k)
    +\frac{N}{2}F^2_{N_{st}}(\lambda_{N/2})
    =O(N^2)\,.
\end{align*}
Thus, by symmetry, $2\pi$-periodicity, and boundedness of $f$, we obtain, similarly as above,
\begin{align}
    &\frac{16\pi^2N_{st}^2}{s^2t^2N^5}
    \sum_{h,j=1}^{\lfloor N/2\rfloor}
    f_{[1]}(\lambda_h)f_{[1]}(\lambda_j)
    T_{N_{st},1}(h,j)
    \notag\\
    &\sim
    \frac{8\pi^2}{3\max^2(s,t)N}
    \bigg[
    2\min(s,t)
    +\frac{(2\min(s,t)-1)^3}{\min^2(s,t)}
    \bigg]
    \sum_{h=0}^{N-1}
    f_{[1]}(\lambda_h)
    \widetilde{\sigma}_{N_{st},1}
    (f_{[1]},\lambda_h)
    \notag\allowdisplaybreaks\\
    &\!\stackrel{N\to\infty}{\longrightarrow}\,
    \frac{4\pi k_1(\min(s,t))}
    {3\max^2(s,t)}
    \int^\pi_{-\pi}f^4(\lambda)\,\mathrm{d}\lambda,
    \label{eq48}
\end{align}
with $k_1\colon\mathbb{D}\to\R$ defined by
\begin{align}
    k_1(s)
    \coloneqq
    2s+\frac{(2s-1)^3}{s^2},
    \qquad s\in\mathbb{D}.
    \label{eq49}
\end{align}

\medskip

\noindent{\bf Step 3.3.} Next, consider the quantity $T_{N_{st},2}$ in \eqref{eq40}. By arguments similar to those above, together with Euler's formula, elementary trigonometric identities, and the fact that $s,t\in\mathbb{D}\subset[1/2,1]$, we obtain
\begin{align}
    &\frac{1}{N}\sum_{k=0}^{N-1}
    F_{N_{st}}(\lambda_{k-1})
    \big[
    F_{N_{st}}(\lambda_k)
    +
    F_{N_{st}}(\lambda_{k+1})
    \big]
    \notag\\[0.5ex]
    &=
    \sum_{\substack{|\ell|,|m|<N_{st}\\
    \ell+m\in\{-N,0,N\}}}
    \Big(1-\frac{|\ell|}{N_{st}}\Big)
    \Big(1-\frac{|m|}{N_{st}}\Big)
    \Big[1+\e^{-\imag\lambda_m}\Big]
    \e^{\imag\lambda_\ell}
    \notag\allowdisplaybreaks\\[0.5ex]
    \begin{split}
    &=
    2+\frac{2}{N_{st}^2}
    \bigg(
    (2N_{st}-N)
    \sum_{\ell=1}^{2N_{st}-N-1}
    \ell
    \big[
    \cos(\lambda_{N_{st}-\ell})
    +
    \cos(2\lambda_{N_{st}-\ell})
    \big]\\[0.5ex]
    &\qquad\qquad\qquad\quad
    +
    \sum_{\ell=2N_{st}-N}^{N_{st}-1}
    \ell^2
    \big[
    \cos(\lambda_{N_{st}-\ell})
    +
    \cos(2\lambda_{N_{st}-\ell})
    \big]
    \bigg).
    \end{split}
    \label{eq50}
\end{align}
For the first sum in \eqref{eq50}, we have
\begin{align}
    \sum_{\ell=1}^{2N_{st}-N-1}\!\!
    \ell
    \big[
    \cos(\lambda_{N_{st}-\ell})
    +
    \cos(2\lambda_{N_{st}-\ell})
    \big]
    \sim
    \frac{N^2(2\min(s,t)-1)}{4\pi}
    \Big[
    2\sin(2\pi\min(s,t))
    +
    \sin(4\pi\min(s,t))
    \Big].
\end{align}
For the second sum in \eqref{eq50}, we obtain
\begin{align*}
    &\sum_{\ell=2N_{st}-N}^{N_{st}-1}
    \ell^2
    \big[
    \cos(\lambda_{N_{st}-\ell})
    +
    \cos(2\lambda_{N_{st}-\ell})
    \big]
    \notag\\[1ex]
    &\sim
    \frac{N^3}{8\pi^3}
    \int_{2\pi(2\min(s,t)-1)}^{2\pi\min(s,t)}
    x^2
    \big[
    \cos(2\pi\min(s,t)-x)
    +
    \cos\big(2(2\pi\min(s,t)-x)\big)
    \big]
    \,\mathrm{d}x
    \notag\allowdisplaybreaks\\[1ex]
    &=
    \frac{N^3}{32\pi^3}
    \bigg[
    20\pi\min(s,t)
    -16\pi(2\min(s,t)-1)
    \cos(2\pi\min(s,t))
    \notag\\
    &\qquad\qquad\;
    -4\pi(2\min(s,t)-1)
    \cos(4\pi\min(s,t))
    -16\pi^2(2\min(s,t)-1)^2
    \sin(2\pi\min(s,t))
    \notag\\
    &\qquad\qquad\;
    -8\pi^2(2\min(s,t)-1)^2
    \sin(4\pi\min(s,t))
    +8\sin(2\pi\min(s,t))
    +
    \sin(4\pi\min(s,t))
    \bigg].
\end{align*}
Combining the two asymptotic expressions, the terms involving $(2\min(s,t)-1)^2\sin(2\pi\min(s,t))$ and $(2\min(s,t)-1)^2\sin(4\pi\min(s,t))$ cancel. Hence,
\begin{align}
    &\frac{1}{N}\sum_{k=0}^{N-1}
    F_{N_{st}}(\lambda_{k-1})
    \big[
    F_{N_{st}}(\lambda_k)
    +
    F_{N_{st}}(\lambda_{k+1})
    \big]
    \sim
    \frac{N\,k_2(\min(s,t))}
    {16\pi^3\min^2(s,t)},
\end{align}
where $k_2\colon\mathbb D=[1/2,1]\to\mathbb{R}$ is defined by
\begin{align}
    k_2(s)
    \coloneqq
    20\pi s
    -4\pi(2s-1)
    \big[
    \cos(4\pi s)
    +
    4\cos(2\pi s)
    \big]
    +
    8\sin(2\pi s)
    +
    \sin(4\pi s).
    \label{eq51}
\end{align}
This function is non-negative, continuous, and $k_2(1)=0$. Finally, we obtain for $T_{N_{st},2}(h,j)$ in \eqref{eq40},
\begin{align}
    \frac{16\pi^2N_{st}^2}{s^2t^2N^5}
    \sum_{h,j=1}^{\lfloor N/2\rfloor}
    f_{[1]}(\lambda_h)f_{[1]}(\lambda_j)
    T_{N_{st},2}(h,j)
    &\sim
    \frac{k_2(\min(s,t))}
    {2\pi\min^2(s,t)\max^2(s,t)N}
    \sum_{h=0}^{N-1}
    f_{[1]}(\lambda_h)
    \widetilde{\sigma}_{N_{st},2}
    (f_{[1]},\lambda_h)
    \notag\allowdisplaybreaks\\[1ex]
    &\!\stackrel{N\to\infty}{\longrightarrow}\,
    \frac{k_2(\min(s,t))}
    {4\pi^2\min^2(s,t)\max^2(s,t)}
    \int_{-\pi}^{\pi}f^4(\lambda)\,\mathrm{d}\lambda.
    \label{eq52}
\end{align}

\medskip

\noindent{\bf Step 3.4.}  Next, we consider the quantities in $T_{N_{st},3}(h,j)$ and $T_{N_{st},4}(h,j)$ in \eqref{eq41} and \eqref{eq42}, respectively, which have a similar structure. Up to phase shifts, the corresponding terms in $N_{st}^2T_{N_{st},3}(h,j)$ and $N_{st}^2T_{N_{st},4}(h,j)$ are of the type $N_{st}^2F_{N_{st}}(\lambda_{h-j})F_{N_{st}}(\lambda_{h+j})$ or $N_{st}^2F^2_{N_{st}}(\lambda_{h+j})$. Using the properties of the Fejér kernel together with Lemma \ref{Lem2} yields
\begin{align}
    &N_{st}^2
    \sum_{h,j=1}^{\lfloor N/2\rfloor}
    F_{N_{st}}(\lambda_{h-j})
    F_{N_{st}}(\lambda_{h+j})
    \,\leq\,
    NN_{st}^3
    +
    N_{st}^2
    \sum_{\substack{h,j=1\\h\neq j}}^{\lfloor N/2\rfloor}
    F_{N_{st}}(\lambda_{h-j})
    F_{N_{st}}(\lambda_{h+j})
    \,=\, O(N^4).
\end{align}
Similarly,
\begin{align}
    N_{st}^2
    \sum_{h,j=1}^{\lfloor N/2\rfloor}
    F^2_{N_{st}}(\lambda_{h+j})
    \,\leq\,
    N^4
    \sum_{h,j=1}^{\lfloor N/2\rfloor}
    \frac{1}{(h+j)^4}
    =O(N^4).
\end{align}
The same bounds hold in the presence of the additional phase shifts occurring in \eqref{eq41} and \eqref{eq42}. Hence,
\begin{align}
    N_{st}^2
    \sum_{h,j=1}^{\lfloor N/2\rfloor}
    \big(
    |T_{N_{st},3}(h,j)|
    +
    |T_{N_{st},4}(h,j)|
    \big)
    =O(N^4)
    =o(N^5).
    \label{eq53}
\end{align}

\medskip

\noindent{\bf Step 3.5.}  Finally, we study the asymptotic behavior of the remaining term $R^{sstt}_N(h,j)$ in \eqref{eq38}. As an illustration, consider its first summand. By using elementary transformations, the bound $D^+_K(x) \leq K$ and the one in Lemma \ref{Lem1}, and arguments similar to those employed above, we have
\begin{align}
    &\sum_{h,j=1}^{\lfloor N/2\rfloor}
    \big|
    D^+_{N_{st}}(\lambda_{h+j})D^+_{N_{st}}(\lambda_{-h+j-1})D^+_{N_{st}}(\lambda_{h-j-1})D^+_{N_{st}}(\lambda_{-h-j+2})
    \big|\notag\\
    &\leq
    N_{st}^2
    \sum_{h,j=1}^{\lfloor N/2\rfloor}
    \big|
    D^+_{N_{st}}(\lambda_{h+j})D^+_{N_{st}}(\lambda_{-h+j-1})
    \big|
    \notag\allowdisplaybreaks\\
    &\leq
    N_{st}^3\Bigg(
    \sum_{h=1}^{\lfloor N/2\rfloor}
    \big|
    D^+_{N_{st}}(\lambda_{2h-1})
    \big|
    +
    \sum_{\substack{h,j=1\\h\neq j-1}}^{\lfloor N/2\rfloor}
    \frac{1}{h+j}\Bigg)
    \notag\\
    &=O\big(N^4\log^2N\big).
\end{align}
The same arguments apply analogously to the remaining summands in $R^{sstt}_N(h,j)$, which yields
\begin{align}
    \sum_{h,j=1}^{\lfloor N/2\rfloor}
    |R^{sstt}_N(h,j)|
    =
    O\big(N^4\log^2N\big)
    =
    o(N^5).
    \label{eq54}
\end{align}

Altogether, the covariance term of interest is asymptotically equivalent to the sum of the terms in \eqref{eq43} and \eqref{eq44}. For \eqref{eq43}, we obtained the limit in \eqref{eq45}. By the decomposition in \eqref{eq38}, the contributions of $T_{N_{st},1}$ and $T_{N_{st},2}$ are given by \eqref{eq48} and \eqref{eq52}, respectively, while the contributions of $T_{N_{st},3}$, $T_{N_{st},4}$, and $R^{sstt}_N$ are asymptotically negligible by \eqref{eq53} and \eqref{eq54}. Combining \eqref{eq45}, \eqref{eq48}, \eqref{eq52}, \eqref{eq53}, and \eqref{eq54}, and by the definitions of $N_s$, $N_t$, and $N_{st}$, we therefore obtain
\begin{align}
N\Cov\!\big(
\widetilde \Psi^{\circ}_N(s),
\widetilde \Psi^{\circ}_N(t)
\big)
&=
\frac{16\pi^2}{N}
\sum_{h,j=1}^{\lfloor N/2\rfloor}
f_{[1]}(\lambda_h)f_{[1]}(\lambda_j)
\Cov\!\big(
I^s_{N,\varepsilon}(\lambda_h)I^s_{N,\varepsilon}(\lambda_{h-1}),
I^t_{N,\varepsilon}(\lambda_j)I^t_{N,\varepsilon}(\lambda_{j-1})
\big)
\notag\\
&\stackrel{N\to\infty}{\longrightarrow}
\tau_{f^2}^2
\left(
\frac{1}{\max(s,t)}
+
\frac{k(\min(s,t))}{4\max^2(s,t)}
\right),\label{eq55}
\end{align}
where $k\colon\mathbb D=[1/2,1]\to\mathbb R$ is defined by
\begin{align}
    k(s)
    &\coloneqq
    \frac{k_1(s)}{3}
    +
    \frac{k_2(s)}{16\pi^3s^2},
    \qquad s\in\mathbb D,
\end{align}
with $k_1$ and $k_2$ defined in \eqref{eq49} and \eqref{eq51}, respectively, which equals the function $k$ in \eqref{eq10}. Since the process in \eqref{eq9} has the covariance structure in \eqref{eq55}, the claim is verified \hfill\qed

\subsection{Proof of Proposition \ref{Prop2}}

From the definition of the sequential periodogram $I_N^s$ in \eqref{eq3}, where
$N_s=\lfloor sN\rfloor$ and $s\in\mathbb D=[1/2,1]$, it follows that
\[
    \big(I_N^s(\lambda)\big)^2
    =
    \frac{1}{4\pi^2s^2N^2}
    \sum_{m,n,o,p=1}^{N_s}
    X_mX_nX_oX_p
    \e^{-\imag(m-n+o-p)\lambda},
    \qquad \lambda\in[-\pi,\pi].
\]
and since $\int_{-\pi}^{\pi}\e^{-\imag j\lambda}\,\mathrm{d}\lambda=2\pi\delta_{j0}$,
\[
    \Psi'_N(s)
    =
    \int_{-\pi}^{\pi}\big(I_N^s(\lambda)\big)^2\,\mathrm{d}\lambda
    =
    \frac{1}{2\pi s^2N^2}
    \sum_{\substack{m,n,o,p=1\\ m-n=p-o}}^{N_s}
    X_mX_nX_oX_p.
\]
Further, since $(X_k)$ is a Gaussian white noise with unit variance, Isserlis's theorem yields
\begin{align*}
    \Exp\big(\Psi'_N(s)\big)
    &=
    \frac{1}{2\pi s^2N^2}
    \sum_{\substack{m,n,o,p=1\\ m-n=p-o}}^{N_s}
    \big(
        \delta_{mn}\delta_{op}
        +\delta_{mo}\delta_{np}
        +\delta_{mp}\delta_{no}
    \big)
    \\[1ex]
    &=
    \frac{1}{2\pi s^2N^2}\,N_s(2N_s+1).
\end{align*}

For the second moment, Isserlis's theorem decomposes the eighth-order Gaussian moment
into the $7!!=105$ possible pairings; cf.\ Step~2 in Section~\ref{Sec42}.
Under the constraints $m-n=p-o$ and $q-r=u-v$, several of these pairings have
fewer free indices and therefore give rise, after summation over the admissible
indices, to asymptotically negligible contributions. For $s\geq t$, we thus have
\begin{align*}
    N\Exp\!\left(\Psi'_N(s)\Psi'_N(t)\right)
    &=
    \frac{1}{4\pi^2s^2t^2N^3}
    \sum_{\substack{m,n,o,p=1\\ m-n=p-o}}^{N_s}~
    \sum_{\substack{q,r,u,v=1\\ q-r=u-v}}^{N_t}
    \Exp\!\big(
        X_mX_nX_oX_pX_qX_rX_uX_v
    \big)
    \\[1ex]
    &=
    \frac{1}{4\pi^2s^2t^2N^3}
    \bigg[
        N_sN_t(2N_s+1)(2N_t+1)
        +32N_sN_t^2
        +\frac{16}{3}N_t^3
        +O(N_t^2)
    \bigg].
\end{align*}
Here, the fraction in the term $\frac{16}{3}N_t^3$ arises from pairings whose associated
index sums involve sums of squares. Consequently,
for $s\geq t$,
\begin{align*}
    N\Cov\big(\Psi'_N(s),\Psi'_N(t)\big)
    &=
    \frac{1}{4\pi^2s^2t^2N^3}
    \bigg[
        32N_sN_t^2
        +\frac{16}{3}N_t^3
        +O(N_t^2)
    \bigg]
    \\[1ex]
    &\stackrel{N \to \infty}{\longrightarrow}
    \frac{4}{\pi^2}
    \left(
        \frac{2}{s}
        +\frac{t}{3s^2}
    \right),
\end{align*}
where we used $N_s/N\to s$ and $N_t/N\to t$. The case $t\geq s$ follows by
symmetry of the covariance, which completes the proof.

\appendix

\section{Auxiliary Results}\label{Sec5}

In the following, we collect properties of the upper Dirichlet kernel $D^+_K$, the Fejér kernel $F_K,$ and the Fejér mean $\sigma_K$ (see Section \ref{Sec4} for their definitions) that are used in the proofs of our results.

\begin{lemma}\label{Lem1} For any $K,N\in\N$ and $\xi\ge 1$, it holds that
\begin{align*}
    \sum_{h=1}^{N-1} \big|D^+_K(\lambda_h)\big|^\xi
    \leq N^\xi
    \begin{cases}
        ~1 + (\xi-1)^{-1}, & \xi>1,\\
        ~1 + \log N,      & \xi=1.
    \end{cases}
\end{align*} 
\end{lemma}

\begin{proof} From the definition of $D^+_K$ and the geometric-sum identity, it follows that
\[
    D_K^+(x)
    =
    \sum_{j=1}^K \e^{-\imag jx}
    =
    \e^{-\imag x}
    \frac{1-\e^{-\imag Kx}}{1-\e^{-\imag x}}, \qquad x \not\equiv 0 \bmod{2\pi}.
\]
Further, by Euler's formula and the Pythagorean identity, we have $|1-\e^{-\imag x}|^2 = 2(1-\cos(x)).$ The identities $|\e^{-\imag y}| = 1$ and $1-\cos x = 2\sin^2(x/2)$ thus yield the upper bound
\[
    \big|D_K^+(x)\big|
    \,\leq\,
    \frac{2}{|1-\e^{-\imag x}|}
    \,=\,
    \frac{1}{|\sin(x/2)|}
    \,\leq\,
    \frac{\pi}{|x|}, \qquad x \in [-\pi, \pi]\setminus\{0\}.
\]
Thus, due to the $2\pi$-periodicity of $|\sin(x/2)|$, and since $\lambda_h = 2\pi h/N,$ we obtain the upper bound
\[
    \big|D_K^+(\lambda_h)\big|
    \,\leq\,
    \frac{N}{2\min(h,N-h)},
    \qquad 1 \leq h < N,
\]
and therefore,
\begin{align*}
    \sum_{h=1}^{N-1}|D_K^+(\lambda_h)|^\xi
    \,\leq\,
    \frac{N^\xi}{2^\xi}
    \sum_{h=1}^{N-1}
    \frac{1}{\min\{h,N-h\}^\xi}    \,\leq\,
    2^{1-\xi}N^\xi
    \sum_{h=1}^{\lfloor N/2\rfloor}h^{-\xi}.
\end{align*}
If $\xi>1$, then
\[
    \sum_{h=1}^{\lfloor N/2\rfloor}h^{-\xi}
    \,\leq\,
    1+\int_1^\infty x^{-\xi}\,\mathrm{d}x
    \,=\,
    \frac{\xi}{\xi-1},
\]
whereas, if $\xi=1$,
\[
    \sum_{h=1}^{\lfloor N/2\rfloor}h^{-\xi}
    \,\leq\,
    1+\int_1^N \frac{\mathrm{d}x}{x}
    \,=\,
    1+\log N.
\]
Since $2^{1-\xi}\leq 1$ for $\xi\geq 1$, the claim is verified.
\end{proof}

\begin{lemma}\label{Lem2}
\begin{enumerate}
    \item[\textnormal{(a)}] For any $K,N \in \N$,
    \[
        F_K(\lambda_j)
        \leq \frac{N^2}{4Kj^2},
        \qquad 0<|j|\leq \lfloor N/2\rfloor.
    \]
    \item[\textnormal{(b)}] For any $K,N \in \N$ with $K\leq N$,
    \[
        \sum_{k=0}^{N-1}F_K(\lambda_k) = N.
    \]
\end{enumerate}
\end{lemma}

\begin{proof}

\noindent \textnormal{(a)~}Using $K F_K(x)=(\sin(Kx/2)/\sin(x/2))^2$ for $x \not\equiv 0 \pmod{2\pi},$ and the inequality $|\sin(x/2)|^{-1}\leq \pi|x|^{-1}$ for $x\in[-\pi,\pi]\setminus\{0\}$, it holds indeed
\[
    F_K(\lambda_j)
    \leq \frac{1}{K|\sin(\lambda_j/2)|^2}
    \leq \frac{N^2}{4Kj^2},
    \qquad 0<|j|\leq\lfloor N/2\rfloor.
\]

\medskip

\noindent \textnormal{(b)~}Since $\sum_{k=0}^{N-1} \e^{-\imag j\lambda_k} = N$ if $j$ is a multiple of $N,$ and zero otherwise, and as $K\leq N$, only the term $j=0$ in the following sum contributes, yielding
\[
    \sum_{k=0}^{N-1} F_K(\lambda_k)
    =
    \sum_{|j|<K}\Bigl(1-\frac{|j|}{K}\Bigr)
    \sum_{k=0}^{N-1}\e^{-\imag j\lambda_k}
    =
    N.
\]
This verifies the claim.
\end{proof}

\begin{lemma}\label{Lem3}
Let $K=K_N\to\infty,$ and let $g:[-\pi,\pi] \to \R$ have a $2\pi$-periodic extension that is H\"older continuous of order $\eta>0.$ Then,
\[
\max_{0\leq h<N}
\big|\sigma_K(g,\lambda_h)-g(\lambda_h)\big|
=
\begin{cases}
    ~O(K^{-1}N^{1-\eta}), & 0<\eta<1,\\[1mm]
    ~O(K^{-1}\log N), & \eta=1.
\end{cases}
\]
\end{lemma}

\begin{proof}
Let $0 \leq h < N$ be arbitrary. Since $\frac1N\sum_{j=0}^{N-1}F_K(\lambda_{h-j})=1$ by Lemma~\ref{Lem2}~(b), we have
\[
    \sigma_K(g,\lambda_h)-g(\lambda_h) = \frac1N\sum_{j=0}^{N-1}
    \,\big(g(\lambda_j)-g(\lambda_h)\big)F_K(\lambda_{h-j})\,.
\]
Hence, by the triangle inequality and the non-negativity of the Fejér kernel, 
\begin{align*}
    \big|\sigma_K(g,\lambda_h)-g(\lambda_h)\big|
    &\leq
    \frac1N\sum_{j=0}^{N-1}
    \big|g(\lambda_j)-g(\lambda_h)\big|
    F_K(\lambda_{h-j}).
\end{align*}
By extending $g : [-\pi,\pi] \to \R$ $2\pi$-periodically, for each $j\in\{0, \dots, N-1\}$, we can choose $k\equiv j-h \bmod N$ such that $|k|\leq \lfloor N/2\rfloor.$ Then, due to $2\pi$-periodicity of
$g$ and $F_K$, we have
\[
    g(\lambda_j)=g(\lambda_{h+k}),
    \qquad
    F_K(\lambda_{h-j})=F_K(-\lambda_k)=F_K(\lambda_k),
\]
where the last equality follows from symmetry of $F_K$. Consequently, because $g$ is H\"older continuous of order $\eta > 0,$ using the inequality in Lemma \ref{Lem2}~(a), and that $\lambda_k = 2\pi k/N,$ we obtain
\begin{align*}
    \big|\sigma_K(g,\lambda_h)-g(\lambda_h)\big|
    &\lesssim
    \frac1N
    \sum_{0<|k|\leq\lfloor N/2\rfloor}
    |\lambda_k|^\eta F_K(\lambda_k)
    \\
    &\lesssim 
    K^{-1}N^{1-\eta}
    \sum_{0<|k|\leq\lfloor N/2\rfloor}
    |k|^{\eta-2}\allowdisplaybreaks\\
    &\lesssim  
    K^{-1}N^{1-\eta}
    \sum^N_{k=1}\,k^{\eta-2}\,.
\end{align*}
Since the latter sum is $O(1)$ for $\eta < 1,$ and $O(\log(N))$ for $\eta=1,$ and as the latter upper bound is independent of $h$, the claim follows.
\end{proof}

\paragraph{Acknowledgements} This work was supported by TRR 391 (Project-ID 520388526) funded by the Deutsche Forschungsgemeinschaft (DFG, German Research Foundation). 

{\footnotesize
    \bibliographystyle{chicago}
    \bibliography{References_PivSpectral}
}

\end{document}